\documentclass[11pt]{article}

\usepackage{amsmath,amsthm}

\usepackage{amssymb,latexsym}

\usepackage{enumerate}

\newcommand{\BB}{{\cal B}}
\newcommand{\DD}{{\cal D}}
\newcommand{\EE}{{\cal E}}
\newcommand{\FF}{{\cal F}}

\newcommand{\HH}{{\cal H}}
\newcommand{\LL}{{\cal L}}
\newcommand{\MM}{{\cal M}}

\newcommand{\RR}{{\cal R}}
\newcommand{\TT}{{\cal T}}
\newcommand{\VV}{{\cal V}}
\newcommand{\WW}{{\cal W}}

\newcommand{\BM}{{\mathbb M}}
\newcommand{\BN}{{\mathbb N}}
\newcommand{\BR}{{\mathbb R}}

\newcommand{\BBM}{{\mathbf{M}}}

\newcommand{\BBR}{{\mathbf R}}

\newcommand{\BBX}{{\mathbf{X}}}

\newcommand{\fch}{{\mathbf{1}}}

\newtheorem{theorem}{\bf Theorem}[section]
\newtheorem{proposition}[theorem]{\bf Proposition}%[subsection]
\newtheorem{lemma}[theorem]{\bf Lemma}%[subsection]
\newtheorem{corollary}[theorem]{\bf Corollary}%[subsection]

\theoremstyle{definition}
\newtheorem*{definition}{Definition}

\newtheorem{remark}[theorem]{\bf Remark}

\numberwithin{equation}{section}

\begin{document}

\title{Large time behaviour of solutions to parabolic equations
with Dirichlet operators and nonlinear dependence on measure data}
\author{Tomasz Klimsiak and Andrzej Rozkosz}
\date{}
\maketitle
\begin{abstract}
We study large time behaviour of renormalized solutions of the
Cauchy problem for equations of the form $\partial_tu-L u+\lambda
u=f(x,u)+g(x,u)\cdot\mu$, where $L$ is the operator associated
with a regular lower bounded semi-Dirichlet form $\EE$ and $\mu$
is a nonnegative bounded smooth measure with respect to the
capacity determined by $\EE$. We show that under the monotonicity
and some integrability assumptions on $f,g$ as well as some
assumptions on the form $\EE$, $u(t,x)\rightarrow v(x)$ as
$t\rightarrow\infty$ for quasi-every $x$, where $v$ is a solution
of some elliptic equation associated with our parabolic equation.
We also provide the rate convergence. Some examples illustrating
the utility of our general results are given.
\end{abstract}
{\bf Mathematics Subject Classification (2010):} Primary: 35B40,
35K58; Secondary: 60H30.
\smallskip\\
{\bf Keywords:} Semilinear equation; Dirichlet operator, Mesure
data, Large time behaviour of solutions,  Rate of convergence,
Backward stochastic differential equation.

\footnote{T. Klimsiak: Institute of
Mathematics, Polish Academy of Sciences, \'Sniadeckich 8, 00-956
Warszawa, Poland, and Faculty of Mathematics and Computer Science,
Nicolaus Copernicus University, Chopina 12/18, 87-100 Toru\'n,
Poland. E-mail: tomas@mat.umk.pl}
%\quad tel.  +48 56 611 3410.}

%\footnote{e-mail: tomas@mat.uni.torun.pl,\quad tel.  +48 56 611
%3410}

\footnote{A. Rozkosz: Faculty of Mathematics and Computer Science,
Nicolaus Copernicus University, Chopina 12/18, 87-100 Toru\'n,
Poland. E-mail: rozkosz@mat.umk.pl}

\section{Introduction} \label{sec1}

Let $E$ be a locally compact separable metric space,  $m$ an
everywhere dense Borel measure on $E$ and let $L$ be the operator
associated with  a regular lower bounded semi-Dirichlet form
$(B,V)$ on $L^2(E;m)$. The main purpose of the paper is to study
large time behaviour of solutions of the Cauchy problem
\begin{equation}
\label{eq1.1} \left\{\begin{array}{l}
\partial_tu-L u+\lambda u=f(x,u)+g(x,u)\cdot\mu
\quad\mbox{in }(0,\infty)\times E,
\medskip\\
u(0,\cdot)=\varphi\quad\mbox{on }E.
\end{array}
\right.
\end{equation}
In (\ref{eq1.1}), $\varphi:E\rightarrow\BR$,
$f,g:E\times\BR\rightarrow\BR$ are Borel measurable functions, $\mu$ is
a smooth measure with respect to the parabolic capacity determined by
$(B,V)$.

The class of operators corresponding to regular lower bounded
Dirichlet forms is quite large. It contains both local operators
whose model example is the Laplace operator $\Delta$ or Laplace
operator perturbed by the first order operator, as well as
nonlocal operators whose model example is the $\alpha$-Laplace
operator $\Delta^{\alpha/2}$ with $\alpha\in(0,2)$ or
$\alpha$-Laplace operator with variable exponent $\alpha$
satisfying some regularity conditions. Many interesting examples
of operators associated with regular semi-Dirichlet forms are to
be found in \cite{FOT,KR:JFA,KR:CM,MR,O2}. In fact, our methods
also allow to treat equations with operators associated with
quasi-regular forms (see remarks at the end of Section
\ref{sec5}).

As for the data $\varphi,f,g$, we assume that $\varphi\in
L^1(E;m)$, $f,g$ are continuous and monotone in the second
variable $u$ and satisfy mild integrability conditions.   Our
basic assumption on $\mu$ is that it is a smooth measure (with
respect to the capacity associated with $(B,V)$) of class
$\RR^+(E)$, i.e. a positive smooth measure such that
$E_xA^{\mu}_{\infty}<\infty$ for quasi-every (q.e. for short)
$x\in E$, where $A^{\mu}$ is the additive functional of the Hunt
process associated with $(B,V)$ in the Revuz correspondence with
$\mu$. Equivalently, our condition imposed on $\mu$ means that the
potential (associated with $(B,V)$) of $\mu$ is $m$-a.e. finite.
It is known that if  $(B,V)$ is a non-symmetric form, and moreover,
it is transient or  $\lambda>0$,
then $\RR^+(E)$ contains the class $\MM^+_{0,b}(E)$ of positive
bounded smooth measures on $E$ (see Section \ref{sec2}). In general, the inclusion
$\MM^+_{0,b}(E)\subset\RR^+(E)$ is strict (see Section \ref{sec2}). Elliptic
equations with unbounded measures of class $\RR^+(E)$ are considered for instance
in the monograph \cite{MV}; see also Section \ref{sec6}.

Let $v$ be a solution of the elliptic equation
\begin{equation}
\label{eq1.2} -Lv+\lambda v=f(x,v)+g(x,v)\cdot\mu\quad \mbox{in }E.
\end{equation}
Our main result says that under the assumptions on $\varphi,f,g$
mentioned before and some additional mild assumptions on the
semigroup $(P_t)$ and the resolvent $(R_{\alpha})$ associated with
$(B,V)$,
\begin{equation}
\label{eq1.3} \lim_{t\rightarrow\infty}u(t,x)=v(x)
\end{equation}
for q.e. $x\in E$. We also estimate the rate of convergence. Our
main estimate says that for every $q\in(0,1)$ there is $C(q)>0$
such that for  q.e. $x\in E$,
\begin{equation}
\label{eq1.4} |u(t,x)-v(x)|\le
3P_t|\varphi|(x)+3P_t(R_0(|f(\cdot,0)|
+|g(\cdot,0)|\cdot\tilde\mu))(x),\quad t>0.
\end{equation}
The quantities on the right hand-side of (\ref{eq1.4}) can be
estimated for concrete operators $L$. We give some examples in
Section \ref{sec6}.

To our knowledge, in case $L$ is a nonlocal operator, our results
(\ref{eq1.3}), (\ref{eq1.4}) are entirely new. In case $L$ is
local, we generalize the results obtained in  the paper
\cite{K:JEE} in which  $g\equiv1$ and $L$ is  a uniformly elliptic
divergence form operator. Note, however, that in \cite{K:JEE}
systems of equations are treated. We also strengthen slightly the
results of \cite{LP} concerning asymptotic behaviour of nonnegative
solutions
of equations involving Laplace operator $\Delta$ and absorbing
term of the form $h(u)|\nabla u|^2$ with $h$ satisfying the ``sign
condition". Some other results on asymptotic behaviour, which are
not covered by our approach, are to be found in
\cite{Pe1,Pe2,Pe3}. In \cite{Pe2,Pe3} equations involving
Leray-Lions type operators and smooth measure data are considered
while \cite{Pe1} deals with  linear equations with general,
possibly singular, bounded measure  $\mu$. Note that the methods
used in \cite{Pe1,Pe2,Pe3} do not provide estimates between the
parabolic solution and the corresponding stationary solution.

In order to prove (\ref{eq1.3}) and (\ref{eq1.4}), we develop the
probabilistic approach initiated in \cite{K:JEE}. We find
interesting that it provides a unified way of treating a wide
variety of seemingly disparate examples (see Section \ref{sec6}).

Although in the paper we deal mainly with the asymptotic behaviour
for solutions of (\ref{eq1.1}), the first question we treat is the
existence and  uniqueness of solutions of problems (\ref{eq1.1})
and (\ref{eq1.2}). Here our results are also new, but our proofs
rely  on our earlier results proved in \cite{KR:JFA,KR:CM} in case
$g\equiv1$. In fact, in the parabolic case we prove the existence
and uniqueness of solutions to problems involving operators $L_t$
and data $f,g,\mu$ depending on time, i.e. more general then
problem (\ref{eq1.1}). Finally, let us note that in the paper we
consider probabilistic solutions of (\ref{eq1.2}) and
(\ref{eq1.3}) (see Section \ref{sec3} for the definitions). It is
worth pointing out, however, that in the case where $(B^{(t)},V)$
are (non-symmetric) Dirichlet forms, the probabilistic solutions
coincide with the renormalized solutions defined in \cite{KR:NoD}
(in the elliptic case under the additional assumption that $(B,V)$
satisfies the strong sector condition and either $(B,V)$ is
transient or $\lambda>0$). For local operators these renormalized
solutions coincide with the usual renormalized solutions (see
\cite{DMOP,PPP} and also \cite{KR:JEE}).

\section{Preliminaries} \label{sec2}

In the paper $E$ is a locally compact separable metric space,
$E^1=\BR\times E$, $m$ is an everywhere dense Borel measure on $E$
and $m_1=dt\otimes m$ . For $T>0$ we write $E_T=[0,T]\times E$,
$E_{0,T}=(0,T]\times E$. By $\BB_b(E)$ we denote the set of all
real bounded Borel measurable functions on $E$ and by $\BB^+_b(E)$
we denote the subset of $\BB_b(E)$ consisting of all nonnegative
functions. The sets $\BB_b(E^1)$, $\BB^+_b(E^1)$ are defined
analogously.

\subsection{Dirichlet forms}

Let $H=L^2(E;m)$ and let  $(\cdot,\cdot)$ denote the usual inner
product in $H$. We assume that we are given a family
$\{B^{(t)},t\in[0,T]\}$ of regular semi-Dirichlet forms on $H$
with common domain $V\subset H$ (see \cite[Section 1.1]{O2}). We
assume that the forms $B^{(t)}$ are lower bounded and satisfy the
sector condition with constants $\alpha_0\ge0$, $K\ge1$
independent of $t\in[0,T]$. Let us recall that this means that
\[
B^{(t)}_{\alpha_0}(\varphi,\varphi)\ge0,\quad \varphi\in V,
\]
where $B^{(t)}_{\lambda}(\varphi,\psi)=B^{(t)}(\varphi,\psi)
+\lambda(\varphi,\psi)$ for $\lambda\ge0$, and that
\[
|B^{(t)}_{\alpha_0}(\varphi,\psi)|\le
KB^{(t)}_{\alpha_0}(\varphi,\varphi)^{1/2}
B^{(t)}_{\alpha_0}(\psi,\psi)^{1/2},\quad \varphi,\psi\in V
\]
for all $t\in[0,T]$. Without loss of generality, we assume $\alpha_0<1$.
We also assume that $[0,T]\ni t\mapsto
B^{(t)}(\varphi,\psi)$ is Borel measurable for every
$\varphi,\psi\in V$ and there is $c\ge1$ such that
\begin{equation}
\label{eq2.01} c^{-1}B_{\alpha_0}(\varphi,\varphi)\le
B_{\alpha_0}^{(t)}(\varphi,\varphi)\le c B_{\alpha_0}(\varphi,\varphi),\quad
t\in[0,T],\,\varphi\in V,
\end{equation}
where $B(\varphi,\varphi)=B^{(0)}(\varphi,\varphi)$. By putting
$B^{(t)}=B$ for $t\notin[0,T]$, we may and will assume that
$B^{(t)}$ is defined and satisfies (\ref{eq2.01}) for all
$t\in\BR$. As usual, we denote  by $\tilde B^{(t)}$ the symmetric
part of $B^{(t)}$, i.e. $\tilde
B^{(t)}(\varphi,\psi)=\frac12(B^{(t)}(\varphi,\psi)
+B^{(t)}(\psi,\varphi))$.

Note that by the assumption, $V$ is a dense subspace of $H$ and the
form $(B,V)$ is closed, i.e. $V$ is a real Hilbert space with
respect to $\tilde B_1(\cdot,\cdot)$, which is densely and
continuously embedded in $H$. We denote by $\|\cdot\|_V$ the norm
in $V$, i.e. $\|\varphi\|^2_V=B_1(\varphi,\varphi)$, $\varphi\in
V$. We denote by $V'$ the dual space of $V$, and by
$\|\cdot\|_{V'}$  the corresponding norm. We set $\HH=L^2(\BR;H)$,
$\VV=L^2(\BR;V)$, $\VV'=L^2(\BR;V')$ and
\begin{equation}
\label{eq2.3} \|u\|^2_{\VV}=\int_{\BR}\|u(t)\|^2_V\,dt,\qquad
\|u\|^2_{\VV'}=\int_{\BR}\|u(t)\|^2_{V'}\,dt.
\end{equation}
We shall identify  $H$ and its dual $H'$. Then $V\subset H\simeq
H'\subset V'$ continuously and densely, and hence
$\VV\subset\HH\simeq\HH'\subset\VV'$ continuously and densely.
%(osrodkowosc?)
%, and by (B2)(?),
%\[
%\|u\|_{\VV'}\|u\|_{\HH}\le \|u\|_{\VV} \quad ? (potrzebne?)
%\]

For $u\in\VV$, we denote by $\frac{\partial u}{\partial t}$ the
derivative in the distribution  sense of the function $t\mapsto
u(t)\in V$, and we set
\begin{equation}
\label{eq2.4} \WW=\left\{u\in \VV:\frac{\partial u}{\partial t}\in
\VV'\right\}, \qquad \|u\|_{\WW}=\|u\|_{\VV}+\left\|\frac{\partial u}{\partial
t}\right\|_{\VV'}\,.
\end{equation}

We denote by  $\EE$  the time dependent Dirichlet  form associated
with the family $\{(B^{(t)},V),t\in\BR\}$, that is
\begin{equation}
\label{eq2.23} \mathcal{E}(u,v)=\left\{
\begin{array}{l}\langle-\frac{\partial u}{\partial t},v\rangle+\BB(u,v),
\quad u\in\WW,v\in\VV,\smallskip \\
\langle\frac{\partial v}{\partial t}, u\rangle+\BB(u,v),\quad
u\in\VV,v\in\WW,
\end{array}
\right.
\end{equation}
where $\langle\cdot,\cdot\rangle$ is the duality pairing between
$\VV'$ and $\VV$, and
\begin{equation}
\label{eq2.24} \BB(u,v)=\int_{\BR}B^{(t)}(u(t),v(t))\,dt.
\end{equation}
Note that $\EE$ can be identified with some generalized Dirichlet
form  (see \cite[Example I.4.9(iii)]{S}).

Given a time dependent form (\ref{eq2.23}), we define quasi notions
with respect to $\EE$ (exceptional sets, nests, quasi-continuity
as in \cite[Section 6.2]{O2}. Note that by \cite[Theorem
6.2.11]{O2} each element $u$ of $\WW$ has a quasi-continuous
$m_1$-version. We will denote it by $\tilde u$. Quasi-notions with
respect to $(B,V)$ are defined as in \cite[Section 2.2]{O2}.

We denote by $S(E)$ the set of all smooth measures on $E$ with
respect to the form $(B,V)$ (see, e.g., \cite[Section 4.1]{O2} for
the definition).  $S(E^1)$ is the set of all smooth
measures on $E^1$ with respect to $\EE$ (see \cite{K:JFA}), and
$S(E_{0,T})$ is the set of all smooth measures on $E^1$ with support
in $E_{0,T}$. We denote by  $\MM_b(E_{0,T})$ the set of all signed
Borel measures on $E^1$ with support in $E_{0,T}$ such that
$|\mu|(E^1)<\infty$, where $|\mu|$ stand for the total variation
of  $\mu$. $\MM_{0,b}(E_{0,T})$ (resp. $\MM^{+}_{0,b}(E_{0,T})$) is the subset of $\MM_b(E_{0,T})$ consisting of all smooth (resp. smooth nonnegative)
measures. Analogously we define the
classes $\MM_b(E)$, $\MM_{0,b}(E)$, $\MM^{+}_{0,b}(E)$.

We will say that a Borel measure $\mu$ on $E^1$ does not depend on
time if it is of the form
\begin{equation}
\label{eq2.6} \mu=dt\otimes\tilde\mu
\end{equation}
for some Borel measure $\tilde \mu$ on $E$. Since $\tilde\mu(B)
=\mu([0,1]\times B)$ for $B\in \BB(E)$, $\tilde\mu$ is uniquely
determined by $\mu$. From now on, given $\mu$ not depending on time,
we  denote by $\tilde\mu$ the Borel measure on $E$ determined
by (\ref{eq2.6}).

\begin{lemma}
If $\mu\in S(E_{0,T})$ does not depend on time, then
$\tilde{\mu}\in S(E)$.
\end{lemma}
\begin{proof}
Let $\alpha>\alpha_0$ and let Cap denote the capacity associated
with the form $B_{\alpha}$ defined in \cite[Definition 4 in Section 2.1]{O2},
whereas  CAP denote the capacity associated with $\EE$ defined in
\cite[(6.2.18) in Section 6.2]{O2}. It is enough to prove that for every
$A\subset E$, if Cap$(A)=0$ then CAP$([0,T]\times A)=0$. Suppose
that Cap$(A)=0$. Then by \cite[Eq. (2.1.8)]{O2}, for every
$\varepsilon>0$ there exists an open set $U_\varepsilon\subset E$
and $\psi_\varepsilon\in V$ such that $A\subset U_\varepsilon$,
$\psi_\varepsilon\ge 1$ on $U_\varepsilon$ and
\[
B_\alpha(\psi_\varepsilon,\psi_\varepsilon) \le
\mbox{Cap}(U_\varepsilon)\le\varepsilon.
\]
By the above inequality and (\ref{eq2.01}),
\begin{equation}
\label{eq2.7} B^{(t)}_\alpha(\psi_\varepsilon,\psi_\varepsilon)
\le c\varepsilon,\quad t\in \BR.
\end{equation}
Let $f$ be a continuous function on $\BR$ with compact support
such  that $f\ge 1$ on $[-T,2T]$ and let
$\eta_\varepsilon=f\psi_\varepsilon$. Then $\eta_\varepsilon \in
\WW$ and by \cite[(6.2.21)]{O2} and (\ref{eq2.7}),
\[
\mbox{CAP}([0,T]\times A)\le C\left(\left\|\frac{\partial
\eta_\varepsilon}{\partial t}\right\|^2_{L^2(0,T;H)}
+\BB_\alpha(\eta_\varepsilon,\eta_\varepsilon)\right) \le \varepsilon
C'T\left(\left\|\frac{\partial f}{\partial t}\right\|^2_\infty+\|f\|^2_\infty\right),
\]
where $C'>0$ depends only on $c$ and $\alpha$. Since
$\varepsilon>0$ was arbitrary, the desired result follows.
\end{proof}

\subsection{Markov processes and additive functionals}
\label{sec2.2}

In what follows $E\cup\{\partial\}$ is a one-point compactification of $E$.
If $E$ is already compact then we adjoin $\partial$ to $E$ as an
isolated point. When considering Dirichlet forms, we adopt the
convention that every function $f$ on $E$ is extended to
$E\cup\{\partial\}$ by setting $f(\partial)=0$. When considering
time dependent Dirichlet forms,  we adopt the convention that
every function $\varphi$ on $E$ is extended to $E^1$ by setting
$\varphi(t,x)=\varphi(x)$, $(t,x)\in E^1$, and every function $f$
on $E^1$ (resp. $E_{0,T}$) is extended to $E^1\cup\{\partial\}$ by
setting $f(\partial)=0$ (resp. $f(z)=0$ for $z\in
E^1\cup\{\partial\}\setminus E_{0,T})$.

Let $\EE$ be the form defined by (\ref{eq2.23}). By \cite[Theorem
6.3.1]{O2}, there exists a Hunt process
$\BBM=(\Omega,(\FF_t)_{t\ge0}, (\BBX_t)_{t\ge0},(P_z)_{z\in
E^1\cup\{\partial\}})$ with state space $E^1$, life time $\zeta$
and cemetery state $\partial$ associated with $\EE$ in the
resolvent sense, i.e. for every $\alpha>0$ and $f\in
L^2(E^1;m_1)\cap\BB_b(E^1)$ the resolvent of $\BBM$ defined as
\[
\BBR_{\alpha}f(z)=\int^{\infty}_0e^{-\alpha
t}E_zf(\BBX_t)\,dt,\quad z\in E^1,f\in\BB_b(E^1),
\]
is an $\EE$-quasi-continuous $m_1$-version of the resolvent
associated with the form $\EE$.  By \cite[Theorem 6.3.1]{O2}, if
\begin{equation}
\label{eq2.9} \BBX_t=(\tau(t),X_{\tau(t)}),\quad t\ge0,
\end{equation}
is a decomposition of $\BBX$ into the process on $\BR$ and on $E$,
then $\tau$ is the uniform motion  to the right, i.e.
$\tau(t)=\tau(0)+t$, $\tau(0)=s$, $P_z$-a.s. for $z=(s,x)\in E^1$.
Moreover, one can check that if $B^{(t)}=B^{(0)}$ for $t\in\BR$,
then  the process
$\BM^{(0)}=(\Omega,(\FF_{t})_{t\ge0},(X_{t})_{t\ge0},
(P_{0,x})_{x\in E\cup\{\partial\}})$ is a Hunt process with life
time $\xi=\inf\{t\ge0:X_{t}\in\partial\}$ associated with the
form $(B^{(0)},V)$.

Let us recall that an additive functional (AF for short) of $\BBM$
is called {\em natural} if $A$ and $\BBM$ have no common
discontinuities. It is known (see \cite[Section 2]{K:JFA}) that
for every $\mu\in S(E^1)$ there exists a unique positive natural
AF $A$ of $\BBM$ such that $A$ is in the Revuz correspondence with
$\mu$, i.e. for every $m_1$-integrable $\alpha$-coexcessive
function $h$ with $\alpha>0$,
\[
\lim_{\beta\rightarrow\infty}\beta E_{h\cdot m_1}
\int^{\infty}_0e^{-(\alpha+\beta)t}
f(\BBX_t)\,dA_t=\int_{E^1}f(z)h(z)\,\mu(dz),\quad
f\in\BB_b^+(E^1),
\]
where $E_{h\cdot m_1}$ denotes the expectation with respect to
$P_{h\cdot m_1}(\cdot)=\int_{E^1}P_z(\cdot)h(z)\,m_1(dz)$. In what
follows we will denote it by $A^{\mu}$. Conversely, if $A$ is
a positive natural AF of $\BBM$ then modifying the proof of
\cite[Lemma 5.1.7]{FOT} (we replace quasi-notions and facts used
in the proof in \cite{FOT} by the corresponding quasi-notions and
facts from \cite[Sections 2--4]{O2}; for the case of
(non-symmetric) Dirichlet form see also \cite[Theorem 5.6]{O1})
one can show that there exists a smooth measure on $E^1$ such that
$A$ is in the Revuz correspondence with $\mu$.

We set
\[
\RR(E_{0,T})=\left\{\mu:|\mu|\in
S(E_{0,T}),\,E_z\int^{\zeta_{\tau}}_0dA^{|\mu|}_t<\infty \mbox{ for
$m_1$-a.e. }z\in E_{0,T}\right\},
\]
where
\[
\zeta_{\tau} =\zeta\wedge(T-\tau(0)).
\]
By \cite[Proposition 3.4]{K:JFA}, in the definition of
$\RR(E_{0,T})$ one can replace $m_1$-a.e. by q.e. (with respect to
$\EE$). By \cite[Proposition 3.8]{K:JFA}, if $(B,V)$ is a
(non-symmetric) Dirichlet form or, more generally, a
semi-Dirichlet form satisfying the duality condition (see
\cite{K:JFA} for the definition), then
$\MM_{0,b}(E_{0,T})\subset\RR(E_{0,T})$. The inclusion may be
strict (see \cite[Example 5.2]{K:JFA}).

Let $\mu\in S(E)$. Since $\BM^{(0)}$ corresponds to $(B,V)$, by
\cite[Theorem 4.1.16]{O2} there is a unique positive continuous AF
$A^{0,\mu}$ of $\BM^{(0)}$ such that $A^{0,\mu}$ is in the Revuz
correspondence with $\mu$, i.e.
\[
\lim_{\alpha\rightarrow\infty}\alpha
E_{m}\int^{\infty}_0e^{-\alpha t}
f(X_t)\,dA^{0,\mu}_t=\int_{E}f(x)\,\mu(dx),\quad f\in\BB_b^+(E).
\]
We set
\[
\RR(E)=\left\{\mu:|\mu|\in S(E), E_{0,x}\int^{\zeta}_0dA^{0,|\mu|}_t<\infty \mbox{ for
$m$-a.e. }x\in E\right\}.
\]
By  \cite[Lemma 4.2]{KR:JFA}, in the above definition of the class
$\RR(E)$ one can replace $m$-a.e. by q.e. (with respect to
$(B,V)$), and  by \cite[Proposition 3.2]{KR:CM}, if $(B,V)$ is
a transient (non-symmetric) Dirichlet form, then
$\MM_{0,b}(E)\subset\RR(E)$. In general, the inclusion is strict
(see remarks following \cite[Proposition 3.2]{KR:CM}).

While considering elliptic equations and large time behaviour of
parabolic equations, we will assume that
\begin{equation}
\label{eq5.12} B^{(t)}(\varphi,\psi)=B(\varphi,\psi),\quad
\varphi,\psi\in V,\quad t\in\BR.
\end{equation}

\begin{lemma}
\label{lem2.1} Assume \mbox{\rm(\ref{eq5.12})}.
\begin{enumerate}
\item[\rm(i)]For every $s\ge0$ the distribution of
$(X\circ\theta_{\tau(0)},A^{0,\tilde\mu}\circ\theta_{\tau(0)})$
under $P_{s,x}$ is equal to the distribution of
$(X,A^{0,\tilde\mu})$ under $P_{0,x}$.
\item[\rm(ii)]$A^{\mu}=A^{0,\tilde\mu}\circ\theta_{\tau(0)}$.
\end{enumerate}
\end{lemma}
\begin{proof}
(i) We first suppose that $\tilde\mu(dx)=f(x)\,m(dx)$ for some $f\in
L^1(E;m)$. Then $A^{0,\tilde\mu}_t=\int^t_0f(X_r)\,dr$, and hence
$A^{0,\tilde\mu}_t\circ\theta_{\tau(0)}
=\int^t_0f(X_r\circ\theta_{\tau(0)})\,dr$. Therefore (i) follows
from the fact that the distribution of $X$ under $P_{0,x}$ is
equal to the distribution of $X\circ\theta_{\tau(0)}$ under
$P_{s,x}$. Now assume that $\mu$ belongs to the set $S_0(E)$ of
smooth measures of finite energy. Then
$A^{0,\tilde\mu}_t=\int^t_0e^r\,d\tilde A_r$, where $\tilde
A_t=\lim_{n\rightarrow\infty}\tilde A^n_t$ and $\tilde
A^n_t=\int^t_0e^{-r}f_n(X_r)\,dr$ for some $f_n\in L^1(E;m)$ (see
the proof of \cite[Theorem 5.1.1]{FOT} or \cite[Theorem
4.1.10]{O2}). From this and the first part we deduce that (i) is
satisfied for every $\tilde\mu\in S_0(E)$. By \cite[Lemma
4.1.14]{O2}, there exists a nest $\{F_n\}$ such that
$\fch_{F_n}\cdot\tilde\mu\in S_0(E)$ for each $n\in\BN$. Since we
already know that (i) holds for $\tilde\mu$ replaced by
$\fch_{F_n}\cdot\tilde\mu$, applying the monotone convergence
theorem we conclude that it holds for $\tilde\mu$ replaced by
$\fch_{\bigcup^{\infty}_{n=1} F_n}\cdot\tilde \mu$, and hence for
$\tilde\mu$ because the set $E\setminus\bigcup^{\infty}_{n=1}F_n$
is exceptional.
\\
(ii) Let $A=A^{0,\tilde\mu}\circ\theta_{\tau(0)}$. Under
(\ref{eq5.12}) the distribution of $A$ under $P_{s,x}$ is equal to
the distribution of $A^{0,\tilde\mu}$ under $P_{0,x}$. Hence
\[
E_{s,x}\int^{\infty}_0e^{-\alpha t}\,dA_t
=E_{s,x}\int^{\infty}_0e^{-\alpha t}\,
d(A^{0,\tilde\mu}_t\circ\theta_s)
=E_{0,x}\int^{\infty}_0e^{-\alpha t}\,
dA^{0,\tilde\mu}_t=:R_{\alpha}\tilde\mu(x).
\]
One can check that $A$ is a CAF of $\BBM$. Let $\nu$ denote its
Revuz measure. Then for every $f$ of the form $f=\xi g$ with
$\xi\in \BB^+_b(\BR)$, $g\in\BB^+_b(E)$ we have
\begin{align*}
\int_{E^1}f(z)\nu(dz)
&=\lim_{\alpha\rightarrow\infty}\alpha\int_{E^1}\Big(f(z)
E_{z}\int^{\infty}_0e^{-\alpha t}\,dA_t\Big)\,m_1(dz)\\
&=\lim_{\alpha\rightarrow\infty}
\alpha\int_{E^1}\xi(s)g(x)R_{\alpha}\tilde\mu(x)\,ds\,m(dx)\\
&=\int_\BR\xi(s)\,ds\cdot \int_{E}g(x)\tilde\mu(dx)
=\int_{E^1}f(s,x)\,ds\,\tilde\mu(dx).
\end{align*}
Hence $\nu=dt\otimes\tilde\mu=\mu$. Since additive functionals are
uniquely determined by their Revuz measures,  this proves (ii).
\end{proof}

\section{Parabolic PDEs and generalized BSDEs} \label{sec3}

For $t\in[0,T]$ let $L_t$ denote the operator associated with the
form $(B^{(t)},V)$, i.e.
\[
D(L_t)=\{u\in V:v\mapsto B^{(t)}(u,v)\mbox{ is continuous with respect
to $(\cdot,\cdot)^{1/2}_H$ on $V$}\}
\]
and
\begin{equation}
\label{eq2.14} (-L_t\varphi,\psi)=B^{(t)}(\varphi,\psi),\quad
\varphi\in D(L_t), \psi\in V
\end{equation}
(see \cite[Proposition I.2.16]{MR}). Suppose we are given
measurable functions $\varphi:E\rightarrow\BR$,
$f,g:E_T\times\BR\rightarrow\BR$ and $\mu\in\RR(E_{0,T})$.  In
this section we consider the following Cauchy problems with
terminal and initial conditions:
\begin{equation}
\label{eq3.2} \partial_tu+L_tu= -f(t,x,u) -g(t,x,u)\cdot\mu,\qquad
u(T)=\varphi
\end{equation}
and
\begin{equation}
\label{eq3.3} \partial_tu-L_{t}u= f(t,x,u) +g(t,x,u)\cdot\mu,\qquad
u(0)=\varphi.
\end{equation}

\begin{definition}
Let $z\in E_T$. We say that a pair $(Y^z,M^z)$ is a solution of
the BSDE
\begin{equation}
\label{eq3.1} Y^z_{t}= \varphi(\BBX_{\zeta_{\tau}})
+\int_{t\wedge\zeta_{\tau}}^{\zeta_{\tau}} f(\BBX_r,Y^z_{r})\,dr
+\int_{t\wedge\zeta_{\tau}}^{\zeta_{\tau}} g(\BBX_r,
Y^z_{r})\,dA^{\mu}_{r}
-\int_{t\wedge\zeta_{\tau}}^{\zeta_{\tau}}dM^z_{r}, \quad t\ge0,
\end{equation}
on the space $(\Omega,\FF,P_z)$ if
\begin{enumerate}
\item[(a)]$Y^z$ is an $(\FF_t)$-progressively measurable process
of class D under $P_z$, $M^z$ is an $(\FF_t)$-martingale under
$P_z$ such that $M^z_0=0$,

\item[(b)]$\int^{\zeta_{\tau}}_0|f(\BBX_t,Y^z_t)|\,dt<\infty$,
$\int^{\zeta_{\tau}}_0 |g(\BBX_t,Y^z_t)|\,d|A^{\mu}|_t<\infty$,
$P_z$-a.s. (Here $|A^{\mu}|_t$ denotes the total variation of the
process $A^{\mu}$ on $[0,t]$),

\item[(c)] Eq. (\ref{eq3.1}) is satisfied $P_z$-a.s.
\end{enumerate}
\end{definition}
%K:\\
%Mamy
%\[
%\int^{\zeta_{\tau}}_t\bar f(\BBX_r,Y_r)\,dr
%$=\int^{\zeta_{\tau}}_t\bar f(\BBX_r,Y_r)\,dr
%\]
%KK\\

Let us recall that a c\`adl\`ag $(\FF_t)$-adapted process $Y$ is
of Doob's class D under $P_z$ if the collection
$\{Y_{\tau}:\tau\in\TT\}$, where $\TT$ is the set of all finite
valued $(\FF_t)$-stopping times, is uniformy integrable under
$P_z$. Let  $\LL^1(P_z)$ denote the space of c\`adl\`ag
$(\FF_t)$-adapted processes $Y$ with finite norm
\[
\|Y\|_{z,1}=\sup\{E_z|Y_{\tau}|:\tau\in\TT\}.
\]
It is known that $\LL^1(P_z)$ is complete (see \cite[p. 90]{DM}).
Moreover, if processes $Y^n$ are of class D and $Y^n\rightarrow Y$
in $\LL^1(P_z)$, then $Y$ is of class D. To see this, let us fix
$\varepsilon>0$ and choose $n$ so that
$\|Y^n-Y\|_{z,1}\le\varepsilon/2$. Since the family
$\{Y^{n}_{\tau}\}$ is of class D, there exists $\delta>0$ such
that if $P_z(A)<\delta$, then
$\int_A|Y^{n}_{\tau}|\,dP_z<\varepsilon/2$. It follows that if
$P_z(A)<\delta$ then for every finite $(\FF_t)$-stopping time
$\tau$,
\[
\int_A|Y_{\tau}|\,dP_z\le E_z|Y^n_{\tau}-Y_{\tau}|
+\int_A|Y^n_{\tau}|\,dP_z\le\varepsilon,
\]
which shows that $\{Y_{\tau}\}$ is uniformly integrable (see
\cite[Theorem I.11]{Pr}).

To simplify notation, in what follows we write
\[
f_{u}(t,x):=f(t,x, u(t,x)), \quad g_{u}(t,x) :=g(t,x,u(t,x)).
\]

\begin{definition}
(a) We say that $u:E_{0,T}\rightarrow\BR$ is a solution of
problem (\ref{eq3.2}) if $f_{u}\cdot m\in\RR(E_{0,T})$,
$g_u\cdot\mu\in\RR(E_{0,T})$ and for q.e. $z\in E_{0,T}$\,,
\begin{equation}
\label{eq4.1} u(z)=E_{z} \left(\varphi(\BBX_{\zeta_{\tau}})
+\int_{0}^{\zeta_{\tau}}f_u(\BBX_{t})\,dt
+\int_{0}^{\zeta_{\tau}}g_u(\BBX_{t})\,dA_{t}^{\mu}\right).
\end{equation}
(b) We say that $u:[0,T)\times E\rightarrow\BR$ is a solution of
problem (\ref{eq3.3}) if $\bar u$ defined as
\[
\bar u(t,x):=u(T-t,x),\quad(t,x)\in E_{0,T}\,,
\]
is a solution of the Cauchy problem with terminal condition of the
form
\begin{equation}
\label{eq3.39}
\partial_t\bar u+L_{T-t}\bar u= -f(T-t,x,\bar u)
-g(T-t,x,\bar u)\cdot(\mu\circ\iota_T^{-1}),\quad \bar
u(T)=\varphi,
\end{equation}
where $\iota_T:E_T\rightarrow E_T$, $\iota_T(t,x)=(T-t,x)$.
\end{definition}

\begin{remark}
\label{rem3.1} If equation (\ref{eq3.39}) has the uniqueness
property (i.e. has a unique solution $v_T$ for every $T>0$), then
for every $a>0$,
\begin{equation}
\label{eq3.40} \bar v_T(t,x)=v_T(T-t,x)=v_{T+a}(T+a-t,x)=\bar
v_{T+a}(t,x),\quad (t,x)\in [0,T)\times E.
\end{equation}
To see this, let us write $f^T_{v_T}(x,t):=f(T-t,x,v_T(t,x))$,
$g^T_{v_T}(x,t):=f(T-t,x,v_T(t,x))$. With this notation,
\[
\frac{\partial v_T}{\partial t}+L_{T-t} v_T =f^T_{v_T}
+g^T_{v_T}\cdot(\mu\circ \iota^{-1}_T), \qquad v_T(T)=\varphi
\]
and
\begin{equation}
\label{eq3.34} \frac{\partial v_{T+a}}{\partial t}+L_{T+a-t}
v_{T+a} = f^{T+a}_{v_{T+a}}+g^{T+a}_{v_{T+a}} \cdot(\mu\circ
\iota^{-1}_{T+a}),\qquad v_{T+a}(T+a)=\varphi.
\end{equation}
Of course, (\ref{eq3.40}) will be proved once we show that
\begin{equation}
\label{eq3.37} v_T(t,x)=v_{T+a}(a+t,x),\quad (t,x)\in E_{0,T}.
\end{equation}
It is known (see \cite[p. 1213]{K:JFA}) that there exists a
generalized nest $\{F_n\}$ on $E_{0,T+a}$ such that
$\Phi^{n,T+a}:=\fch_{F_n}\cdot(f^{T+a}_{v_{T+a}}+
g^{T+a}_{v_{T+a}} \cdot(\mu\circ\iota^{-1}_{T+a}))\in
S_0(E_{0,T+a})$ for each $n\in\BN$. Let $v^n_{T+a}$ denote the
solution of the linear equation
\begin{equation}
\label{eq3.35} \frac{\partial v^n_{T+a}}{\partial t}+L_{T+a-t}
v^n_{T+a}=\Phi^{n,T+a},\qquad v^n_{T+a}(T+a)=\varphi,
\end{equation}
and let
\begin{equation}
\label{eq3.38} v^n_{T+a,a}(t,x):=v^n_{T+a}(a+t,x),\quad (t,x)\in
(-a,T]\times E.
\end{equation}
By \cite[Theorem 3.7]{K:JFA}, $v^n_{T+a}$ is a weak solution of
(\ref{eq3.35}). Therefore making a simple change of variables
shows that $v^n_{T+a,a}$ is a weak solution of the linear equation
\begin{equation}
\label{eq3.36} \frac{\partial v^n_{T+a,a}}{\partial t}+L_{T-t}
v^n_{T+a,a}=\mathbf{1}_{F_n}^a\cdot(f^{T}_{v_{T+a,a}}+
g^T_{v_{T+a,a}}\cdot (\mu\circ\iota^{-1}_T)), \quad
v^n_{T+a,a}(T)=\varphi,
\end{equation}
where $\mathbf{1}^a_{F_n}(t,x)=\mathbf{1}_{F_n}(t+a,x)$. Using the
probabilistic representation of the solution of (\ref{eq3.35}) and
the fact that $\{F_n\}$ is a nest, one can easily show that
$v^n_{T+a}\rightarrow v_{T+a}$ pointwise as $n\rightarrow\infty$.
Similarly,  using the probabilistic representation of the solution
of (\ref{eq3.36}) one can show that $v^n_{T+a,a}$ converges
pointwise as $n\rightarrow\infty$ to the solution of
(\ref{eq3.34}), that is to $v_T$. This and (\ref{eq3.38}) imply
(\ref{eq3.37}).
\end{remark}

In the rest of this section we say that some property is satisfied
quasi-everywhere (q.e. for brevity) if the set of those $z\in E^1$
for which it does not hold is exceptional with respect to the form
$\EE$.

In what follows we  say that a Borel measurable
$F:E_{0,T}\rightarrow\BR$ is $\mu$-quasi-integrable ($F\in
qL^1(E_{0,T};\mu)$ in notation) if
$P_z(\int^{\zeta_{\tau}}_0|F(\BBX_t)|\,dA^{\mu}_t<\infty)=1$ for
q.e. $z\in E_{0,T}$.

Let us remark that if $\mu=m_1$, then $A^{\mu}_t=t$, $t\ge0$, so
$m_1$-quasi-integrability coincides with the notion of
quasi-integrability considered in \cite[Section 5]{K:JFA}) (see
also \cite[Section 2]{K:AMPA}).

Our basic assumptions on the data are the following.
\begin{enumerate}
\item[(P1)]$\varphi\in L^1(E;m)$, $\mu\in\RR^{+}(E_{0,T})$.

\item[(P2)]$f(\cdot,\cdot,y),g(\cdot,\cdot,y)$
are measurable for every $y\in\BR$ and $f(t,x,\cdot),g(t,x,\cdot)$
are continuous for every $(t,x)\in E_{0,T}$.

\item[(P3)]There is $\alpha\in\BR$ such that
$(f(t,x,y)-f(t,x,y'))(y-y')\le \alpha |y-y'|^{2}$ for
all $y,y'\in\BR$ and $(t,x)\in E_{0,T}$.

\item[(P4)]$f(\cdot,\cdot,0)\cdot m_1\in \RR(E_{0,T})$ and
$(t,x)\mapsto f(t,x,y)\in qL^{1}(E_{0,T};m_1)$ for every
$y\in\BR$.

\item[(P5)]$(g(t,x,y)-g(t,x,y'))(y-y')\le0$ for
all $y,y'\in\BR$ and  $(t,x)\in E_{0,T}$.

\item[(P6)]$g(\cdot,\cdot,0)\cdot\mu\in\RR(E_{0,T})$ and
$(t,x)\mapsto g(t,x,y)\in qL^{1}(E_{0,T};\mu)$ for every
$y\in\BR$.
\end{enumerate}

In what follows we denote by $\DD^q(P_z)$, $q>0$,  the space of
all $(\FF_t)$-progressively measurable c\`adl\`ag processes $Y$
such that $E_z\sup_{t\ge0}|Y_t|^q<\infty$.

\begin{theorem}
\label{th3.1}  Assume that \mbox{\rm(P1)--(P6)} are satisfied and
$A^{\mu}$ is continuous.
\begin{enumerate}
\item[\rm(i)]There exists a unique solution $u$ of problem
\mbox{\rm(\ref{eq3.2})}.
\item[\rm(ii)]Let
\[
M_{t}^{z}=E_{z}\left(\varphi(\BBX_{\zeta_{\tau}})
+\int_{0}^{\zeta_{\tau}} f_{u}(\BBX_{s})\,ds
+\int_{0}^{\zeta_{\tau}} g_{u}
(\BBX_s)\,dA_{s}^{\mu}\,\Bigl|\,\FF_{t}\right)-u(\BBX_{0}).
\]
Then there exists a c\`adl\`ag $(\FF_t)$-adapted process $M$ such
that $M_t=M^z_t$, $t\in[0,T]$, $P_z$-a.s. for q.e. $z\in E_{0,T}$,
and for q.e $z\in E_{0,T}$ the pair $(u(\BBX),M)$ is a unique
solution of \mbox{\rm(\ref{eq3.1})} on the space
$(\Omega,\FF,P_z)$. Moreover,  $u(\BBX)\in\DD^q(P_z)$ for
$q\in(0,1)$ and $M$ is a uniformly integrable martingale under
$P_z$ for q.e. $z\in E_{0,T}$. Finally,  for q.e. $z\in E_{0,T}$,
\begin{align}
\label{eq3.33} &E_z\int_{0}^{\zeta_{\tau}} f_{u}(\BBX_t)\,dt
+\int_{0}^{\zeta_{\tau}} g_{u}
(\BBX_t)\,dA_{t}^{\mu}\nonumber\\
&\qquad\le E_z\left(\varphi(\BBX_{\zeta_{\tau}})
+2\int^{\zeta_{\tau}}_0|f(\BBX_t,0)|\,dt
+3\int^{\zeta_{\tau}}_0|g(\BBX_t,0)|\,dA^{\mu}_t\right).
\end{align}
\end{enumerate}
\end{theorem}
\begin{proof}
By using the standard change of variables (see, e.g., the
beginning of the proof of \cite[Lemma 3.1]{BDHPS}), without loss
of generality we may and will assume that $\alpha\le0$ in
condition (P3).

We first prove (ii). The uniqueness of a solution of BSDE
(\ref{eq3.1}) follows from (P3), (P5) and the fact that $\mu$ is
nonnegative.  The proof is standard. We may argue for instance as
in the proof of \cite[Proposition 2.1]{KR:JFA} with obvious
changes. We divide the proof of existence of a solution into two steps. \\
Step 1. Let $\xi=\varphi(\BBX_{\zeta_{\tau}})$,
$f(t,y)=f(\BBX_t,y)$, $g(t,y)=g(\BBX_t,y)$ and let $A$ be a
continuous increasing $(\FF_t)$-adapted process. Assume that
\[
T\cdot\sup_{0\le t\le T}|f(t,0)| +A_{T}\cdot\sup_{0\le t\le
T}|g(t,0)|+|\xi|\le c
\]
$P_z$-a.s. $z\in E_{0,T}$ for some $c>0$, and write  $\bar
f_c(t,y)=f(t,T_c(y))$, $\bar g_c(t,y)=g(t,T_c(y))$, where
\begin{equation}
\label{eq3.32} T_c(y)=((-c)\vee y)\wedge c,\quad y\in \BR.
\end{equation}
Then modifying slightly the proof of \cite[Lemma 2.6]{KR:JFA}, we
show that there exists a unique solution $(Y,M)$ of the BSDE
\begin{align}
\label{eq3.22}
Y_t&=\xi+\int^{\zeta_{\tau}}_{t\wedge\zeta_{\tau}}\bar
f_c(s,Y_s)\,ds +\int^{\zeta_{\tau}}_{t\wedge\zeta_{\tau}}
(\bar g_c(s,Y_s)-g(s,0))\,dA^{\mu}_s\nonumber\\
&\quad +\int^{\zeta_{\tau}}_{t\wedge\zeta_{\tau}}\bar
g_c(s,0)\,dA_s -\int^{\zeta_{\tau}}_{t\wedge\zeta_{\tau}}dM_s,
\quad t\ge0,
\end{align}
on the space $(\Omega,\FF,P_z)$ (for brevity, in our notation we
drop the dependence of $Y,M$ on $z$). Let $\mbox{sgn}(x)=1$ if
$x>0$ and $\mbox{sgn}(x)=-1$ if $x\le0$. By the Meyer-Tanaka
formula (see \cite[p. 216]{Pr}) and the fact that $A^{\mu}$ is
continuous,
\begin{align*}
|Y_t|&\le|Y_{\zeta_{\tau}}|
-\int^{\zeta_{\tau}}_{t\wedge\zeta_{\tau}}
\mbox{sgn}(Y_{s})\,dY_s\\
&=|\xi|+\int^{\zeta_{\tau}}_{t\wedge\zeta_{\tau}}
\mbox{sgn}(Y_{s})(\bar f_c(s,Y_s)-f(s,0))\,ds
+\int^{\zeta_{\tau}}_{t\wedge\zeta_{\tau}}\mbox{sgn}(Y_{s})f(s,0)\,ds\\
&\quad+\int^{\zeta_{\tau}}_{t\wedge\zeta_{\tau}}
\mbox{sgn}(Y_{s})\{(\bar g_c(s,Y_s)-g(s,0))\,dA^{\mu}_s +\bar
g_c(s,0)\,dA_s\} -\int^{\zeta_{\tau}}_{t\wedge\zeta_{\tau}}
\mbox{sgn}(Y_{s-})\,dM_s.
\end{align*}
From this, (\ref{eq3.22}) and (P3), (P5) we get
\begin{equation}
\label{eq3.9} |Y_t|=E_z(|Y_t|\,|\FF_t)\le
E_z\left(|\xi|+\int^{\zeta_{\tau}}_0(|f(s,0)|\,ds
+|\bar g_c(s,0)|\,dA_s)\,\Bigl|\,\FF_t\right)\le c,
\end{equation}
which shows that in fact $(Y,M)$ is a solution of (\ref{eq3.22})
with $\bar f_c$ replaced by $f$ and $\bar g_c$ replaced by $g$.
\\
Step 2. For $n\ge0$, we set  $\xi^n=T_n(\xi)$,
$f_n(t,y)=f(t,y)-f(t,0)+T_n(f(t,0))$ (with $\xi$, $f(t,y)$ defined
in Step 1) and $A^n_t=\int^t_0\fch_{\{A^{\mu}_{r}\le
n\}}\,dA^{\mu}_{r}$. By Step 1, for each $n\ge0$ there exists a
unique solution $(Y^{n},M^{n})$ of the BSDE
\begin{align}
\label{eq3.11}
Y^n_t&=\xi^n+\int^{\zeta_{\tau}}_{t\wedge\zeta_{\tau}}f_n(s,Y^n_s)\,ds
+\int^{\zeta_{\tau}}_{t\wedge\zeta_{\tau}}
(g(s,Y^n_s)-g(s,0))\,dA^{\mu}_s\nonumber\\
&\quad +\int^{\zeta_{\tau}}_{t\wedge\zeta_{\tau}}
T_ng(s,0)\,dA^{n}_s -\int^{\zeta_{\tau}}_{t\wedge\zeta_{\tau}}
dM^n_s,\quad t\ge0
\end{align}
on the space $(\Omega,\FF,P_z)$ (as in Step 1, for brevity, in our
notation we drop the dependence of $(Y^n,M^n)$ on $z$). For $m\ge
n\ge0$, we write $\delta Y=Y^m-Y^n$, $\delta M=M^m-M^n$,
$\delta\xi=\xi^m-\xi^n$. Since $\mu\in\RR^+(E_{0,T})$, $A^m$ is an
increasing process. Therefore using the Meyer-Tanaka formula  we
obtain
\begin{align*}
|\delta Y_t|
&\le|\delta\xi|+\int^{\zeta_{\tau}}_{t\wedge\zeta_{\tau}}
\mbox{sgn}(\delta Y_s)(f_m(s,Y^m_s)-f_n(s,Y^n_s))\,ds\\
&\quad+\int^{\zeta_{\tau}}_{t\wedge\zeta_{\tau}} \mbox{sgn}(\delta
Y_{s})(g(s,Y^m_s)-g(s,Y^n_s))\,dA^{\mu}_s\\
&\quad+\int^{\zeta_{\tau}}_{t\wedge\zeta_{\tau}}\mbox{sgn}(\delta
Y_{s}) \{T_mg(s,0)\,dA^m_s-T_ng(s,0)\,d A^n_s\}
+\int^{\zeta_{\tau}}_{t\wedge\zeta_{\tau}}\mbox{sgn}(\delta
Y_{s-})\,d(\delta M)_s.
\end{align*}
From the above  and (P3), (P5) it follows that
\begin{align*}
|\delta Y_t|&\le
|\delta\xi|+\int^{\zeta_{\tau}}_{t\wedge\zeta_{\tau}}
|T_mf(s,0)-T_nf(s,0)|\,ds
+\int^{\zeta_{\tau}}_{t\wedge\zeta_{\tau}}|T_mg(s,0)-T_ng(s,0)|\,dA^m_s\\
&\quad+\int^{\zeta_{\tau}}_{t\wedge\zeta_{\tau}}
|T_ng(s,0)|\,d(A^m_s-A^n_s)
+\int^{\zeta_{\tau}}_{t\wedge\zeta_{\tau}} \mbox{sgn}(\delta
Y_{s-})\,d(\delta M)_s.
\end{align*}
Hence
\begin{equation}
\label{eq3.10} |\delta Y_t|=E_z(|\delta Y_t||\FF_t)\le
E_z(\Psi^n|\FF_t),\quad t\ge0,
\end{equation}
where
\begin{align*}
\Psi^n&=|\xi|\fch_{\{|\xi|>n\}}
+\int^{\zeta_{\tau}}_0|f(t,0)|\fch_{\{|f(t,0)|>n\}}\,dt\\
&\quad +\int^{\zeta_{\tau}}_0|g(t,0)|\fch_{\{|g(t,0)|>n\}}
\,dA^{\mu}_t +\int^{\zeta_{\tau}}_0|g(t,0)|\,d(A^m_t-A^n_t).
\end{align*}
Observe that from our assumptions on the data $\varphi,f,g,\mu$ it
follows that $E_z\Psi^n\rightarrow0$ as $n\rightarrow\infty$ for
q.e. $z\in E_{0,T}$. By (\ref{eq3.10}), $\|\delta Y\|_{z,1}\le
E_z\Psi^n$, while by  \cite[Lemma 6.1]{BDHPS}, $E_z\sup_{t\le
T}|\delta Y_t|^q\le (1-q)^{-1}(E_z\Psi^n)^q$ for every
$q\in(0,1)$. Since the spaces $\DD^q(P_z)$ and $\LL^1(P_z)$ are
complete, for q.e. $z\in E_{0,T}$ there exists a process $Y^z$
such that $Y^z\in\DD^q(P_z)$ for $q\in(0,1)$, $Y^z$ is of class D
under $P_z$ and
\begin{equation}
\label{eq3.19} \|Y^n-Y^z\|_{1,z}\rightarrow0,\quad E_z\sup_{0\le
t\le \zeta_{\tau}} |Y^n_t-Y^z_t|^q\rightarrow0.
\end{equation}
We have
\begin{align*}
\int^{\zeta_{\tau}}_0|f_n(t,Y^n_t)-f(t,Y^z_t)|\,dt &\le
\int^{\zeta_{\tau}}_0|f(t,Y^n_t)-f(t,Y^z_t)|\,dt\\
&\quad +\int^{\zeta_{\tau}}_0|f(t,0)| \fch_{\{|f(t,0)|>n\}}\,dt.
\end{align*}
Applying the Meyer-Tanaka formula we get (see the proof of
(\ref{eq3.9}))
\[
|Y^n_t|\le E_z\left(|\xi|+\int^{\zeta_{\tau}}_0|f(s,0)|\,ds
+\int^{\zeta_{\tau}}_0|g(s,0)|\,dA^{\mu}_s\,\Bigl|\,\FF_t\right)
=:R_t,\quad t\ge0.
\]
For $k,N\in\BN$, we set
\begin{align*}
\tau_{k,N}&=\inf\left\{t\ge0:R_t\ge k,\int^t_0(|f(s,-k)|+|f(s,k)|)\,ds\right.\\
&\qquad\qquad+\left.\int^t_0(|g(s,-k)|+|g(s,k)|)\,dA^{\mu}_s\ge N\right\}\wedge
\zeta_{\tau}.
\end{align*}
By (\ref{eq3.11}),
\begin{align}
\label{eq3.15} Y^n_{t\wedge\tau_{k,N}}&=E_z\left(Y^n_{\tau_{k,N}}
+\int^{\tau_{k,N}}_{t\wedge\tau_{k,N}}f_n(s,Y^n_s)\,ds\right. \nonumber\\
&\qquad\left.+\int^{\tau_{k,N}}_{t\wedge\tau_{k,N}}
\{g(s,Y^n_s)-g(s,0))\,dA^{\mu}_s
+T_ng(s,0)\,dA^{n}_s\}\,\Bigl|\,\FF_t\right).
\end{align}
From the definition of $\tau_{k,N}$ it follows that
\[
\int^{\tau_{k,N}}_0|f(t,Y^n_t)|\,dt
+\int^{\tau_{k,N}}_0|g(t,Y^n_t)|\,dA^{\mu}_t\le N.
\]
%K:\\
%Calka do  $\tau_{k,N}$ jest rowna calce do $\tau_{k,N}-$.\\
%KK:\\
From this, (P2) and (\ref{eq3.19}) one can deduce that
\begin{equation}
\label{eq3.16} \lim_{n\rightarrow\infty}E_z\int^{\tau_{k,N}}_0
(|f_n(t,Y^n_{t})-f(t,Y^z_{t})|\,dt=0
\end{equation}
and
\begin{equation}
\label{eq3.17} \lim_{n\rightarrow\infty}E_z\int^{\tau_{k,N}}_0
\{|g(t,Y^n_t)-g(t,0)|\,dA^{\mu}_t+|T_ng(t,0)|\,dA^n_t\}=0.
\end{equation}
By Doob's inequality (see, e.g., \cite[Theorem 1.9.1]{LS}) and
(\ref{eq3.19}), for every $\varepsilon>0$ we have
\begin{equation}
\label{eq3.23} \lim_{n\rightarrow\infty}P_x(\sup_{t\le T}
|E_z(Y^n_{\tau_{k,N}}
-Y^z_{\tau_{k,N}}|\FF_t)|>\varepsilon) \\
\le \varepsilon^{-1}\lim_{n\rightarrow\infty} E_z|Y^n_{\tau_{k,N}}
-Y^z_{\tau_{k,N}}|=0.
\end{equation}
Similarly, by (\ref{eq3.16}), (\ref{eq3.17}) and Doob's
inequality,
\begin{equation}
\label{eq3.24} \lim_{n\rightarrow\infty}P_z\left(\sup_{t\le T}
\left|E_z\left(\int^{\tau_{k,N}}_{t\wedge\tau_{k,N}}
(f(s,Y^n_s)-f(s,Y^z_s))\,ds\,\Bigl|\,\FF_t\right)\right|
>\varepsilon\right)=0
\end{equation}
and
\begin{equation}
\label{eq3.25} \lim_{n\rightarrow\infty}P_z\left(\sup_{t\le T}
\left|E_z\left(\int^{\tau_{k,N}}_{t\wedge\tau_{k,N}} (g(s,Y^n_s)
-g(s,Y^z_s))\,dA^{\mu}_s\,\Bigl|\,\FF_t\right)\right|>\varepsilon\right)=0
\end{equation}
for every $\varepsilon>0$. Letting $n\rightarrow\infty$ in
(\ref{eq3.15})  and using (\ref{eq3.23})--(\ref{eq3.25}) we
conclude that
\begin{equation}
\label{eq3.27} Y^z_{t\wedge\tau_{k,N}}=E_z\left(Y^z_{\tau_{k,N}}
+\int^{\tau_{k,N}}_{t\wedge\tau_{k,N}} \{f(s,Y^z_s)\,ds
+g(s,Y^z_s)\,dA^{\mu}_s\}\,\Bigl|\,\FF_t\right).
\end{equation}
We have
\begin{align*}
&\int^{\zeta_{\tau}}_0|f_n(t,Y^n_t)|\,dt
+\int^{\zeta_{\tau}}_0|g(t,Y^n_t)|\,dA^{\mu}_t \\
&\qquad\le \int^{\zeta_{\tau}}_0|f_n(t,Y^n_t)-f_n(t,0)|\,dt
+\int^{\zeta_{\tau}}_0|f_n(t,0)|\,dt\\
&\qquad\quad +\int^{\zeta_{\tau}}_0|g(t,Y^n_t)-g(t,0)|\,dA^{\mu}_t
+\int^{\zeta_{\tau}}_0|g(t,0)|\,dA^{\mu}_t\\
&\qquad=-\int^{\zeta_{\tau}}_0\mbox{sgn}(Y^n_{t})(f_n(t,Y^n_t)-f_n(t,0))\,dt
+\int^{\zeta_{\tau}}_0|f_n(t,0)|\,dt\\
&\qquad\quad
-\int^{\zeta_{\tau}}_0\mbox{sgn}(Y^n_{t})(g(t,Y^n_t)-g(t,0))\,dA^{\mu}_t
+\int^{\zeta_{\tau}}_0|g(t,0)|\,dA^{\mu}_t.
\end{align*}
By the Meyer-Tanaka formula and (\ref{eq3.11}),
\begin{align*}
|\xi^n|-|Y^n_0|&\ge -\int^{\zeta_{\tau}}_0\mbox{sgn}(Y^n_t)
f_n(t,Y^n_t)\,dt -\int^{\zeta_{\tau}}_0\mbox{sgn}(Y^n_{t})
g(t,Y^n_t)\,dA^{\mu}_t\\
&\quad-\int^{\zeta_{\tau}}_0\mbox{sgn}(Y^n_{t})T_ng(t,0)\,dA^n_t
-\int^{\zeta_{\tau}}_0\mbox{sgn}(Y^n_{t-})\,dM_t.
\end{align*}
Hence
\begin{align*}
&E_z\int^{\zeta_{\tau}}_0\{|f_n(t,Y^n_t)|\,dt
+|g(t,Y^n_t)|\,dA^{\mu}_t\} \\
&\qquad\le E_z\left(|\xi^n|+2\int^{\zeta_{\tau}}_0|f(t,0)|\,dt
+3\int^{\zeta_{\tau}}_0|g(t,0)|\,dA^{\mu}_t \right),
\end{align*}
so applying Fatou's lemma and (\ref{eq3.19}) gives
\begin{align}
\label{eq3.20}
&E_z\int^{\zeta_{\tau}}_0\{|f(t,Y^z_t)|\,dt+|g(t,Y^z_t)|\,dA^{\mu}_t\}
\nonumber\\
&\qquad \le E_z\left(|\varphi(\BBX_{\zeta_{\tau}})|
+2\int^{\zeta_{\tau}}_0|f(t,0)|\,dt
+3\int^{\zeta_{\tau}}_0|g(t,0)|\,dA^{\mu}_t\right)<\infty.
\end{align}
Since $f(\cdot,-k),f(\cdot,k)\in qL^1(E_{0,T};m_1)$ and
$g(\cdot,-k),g(\cdot,k)\in qL^1(E_{0,T};\mu)$,
$\tau_{k,N}\rightarrow\tau_k$ as $N\rightarrow\infty$, where
\[
\tau_{k}=\inf\{t\ge0:R_t\ge k\}\wedge \zeta_{\tau}.
\]
Hence $Y^z_{\tau_{k,N}}\rightarrow Y^z_{\tau_k}$ $P_z$-a.s., and
consequently,
\begin{equation}
\label{eq3.28} \lim_{N\rightarrow\infty}
E_z|Y^z_{\tau_{k,N}}-Y^z_{\tau_{k}}|=0
\end{equation}
since $Y^z$ is of class D. Letting $N\rightarrow\infty$ in
(\ref{eq3.27}) and using (\ref{eq3.20}), (\ref{eq3.28}) and Doob's
inequality we obtain
\begin{equation}
\label{eq3.29}
Y^z_{t\wedge\tau_k}=E_z\left(Y^z_{\tau_k}+\int^{\tau_k}_{t\wedge\tau_k}
\{f(s,Y^z_s)\,ds +g(s,Y^z_s)\,dA^{\mu}_s\}\,\Bigl|\,\FF_t\right).
\end{equation}
Since $\tau_k\rightarrow \zeta_{\tau}$ as $k\rightarrow\infty$,
letting $k\rightarrow\infty$ in (\ref{eq3.29}) and repeating
arguments used to prove  (\ref{eq3.29}) we get
\[
Y^z_{t}=E_z\left(\xi+\int^{\zeta_{\tau}}_{t\wedge \zeta_{\tau}}
\{f(s,Y^z_s)\,ds +g(s,Y^z_s)\,dA^{\mu}_s\}\,\Bigl|\,\FF_t\right).
\]
We may now repeat the reasoning following
\cite[(3.6)]{KR:CM} with the process $V$ from \cite{KR:CM} replaced by
$\int^{\cdot}_0g(t,Y^z_t)\,dA^{\mu}_t$ (see also the reasoning
following (\ref{eq5.17}) in the present paper) to prove that the
pair $(Y^z,\tilde M^z)$, where $\tilde M^z$ is a c\`adl\`ag
version of the martingale
\[
t\mapsto E_{z}\left(\varphi(\BBX_{\zeta_{\tau}})
+\int_{0}^{\zeta_{\tau}}\{f(\BBX_{s},Y^z_s)\,ds
+g(\BBX_s,Y^z_s)\,dA_{s}^{\mu}\}\,\Bigl|\,\FF_{t}\right)-\bar u(\BBX_{0}),
\]
is a solution of the BSDE
\begin{align}
\label{eq3.21} Y^z_{t}= \varphi(\BBX_{\zeta_{\tau}})
+\int^{\zeta_{\tau}}_{t\wedge\zeta_{\tau}}\{f(\BBX_s,Y^z_{s})\,ds
+ g(\BBX_s,Y^z_{s})\,dA^{\mu}_{s}\}
-\int^{\zeta_{\tau}}_{t\wedge\zeta_{\tau}}\,d\tilde M^z_{s}, \quad
t\ge0,
\end{align}
on $(\Omega,\FF,P_z)$. Furthermore, by \cite[Remark 3.6]{KR:JFA},
there exists a pair of processes $(Y,M)$ such that
$(Y_t,M_t)=(Y^z,\tilde M^z_t)$, $t\in[0,\zeta_{\tau}]$, $P_z$-a.s.
for q.e. $z\in E_{0,T}$. Let $u(z)=E_zY_0$. Then the argument from
the beginning of the proof of \cite[Theorem 5.8]{K:JFA} shows that
$Y_t=u(\BBX_t)$, $t\in[0,\zeta_{\tau}]$, which implies that $M$ is
a version of the martingale $M^z$ and that $(u(\BBX),M)$ is a
solution of (\ref{eq3.21}) for q.e. $z\in E_{0,T}$. In view of our
convention made at the beginning of Section \ref{sec2.2}, this
means that $(u(\BBX),M)$ is a solution of (\ref{eq3.1}) on the
space $(\Omega,\FF,P_z)$ for q.e. $z\in E_{0,T}$. Of course,
$u(\BBX)\in\DD^q(P_z)$. Furthermore, $M$ is a uniformly integrable
martingale under $P_z$, because under $P_z$ it is a version of the
closed martingale $M^z$. Finally, since we know that
$Y^z_t=u(\BBX)$, $t\in[0,\zeta_{\tau}]$, $P_z$-a.s., (\ref{eq3.33})
follows immediately from (\ref{eq3.20}). This completes the proof
of part (ii) of the theorem.

Part (i) follows from (ii). Indeed, since $\mu\in\RR(E_{0,T})$ and
we know that (\ref{eq3.20}) is satisfied with $Y^z$ replaced by
$u(\BBX)$ and $M$ is a martingale under $P_z$ for q.e. $z\in
E_{0,T}$, putting $t=0$ in (\ref{eq3.1}) and then taking the
expectation shows that $\bar u$ is a solution of (\ref{eq3.2}). To
show that $\bar u$ is unique one can argue as in the proof of
\cite[Theorem 5.8]{K:JFA}.
\end{proof}

\begin{remark} If $g$ does not depend on the last variable $y$, then in
Theorem \ref{th3.1} we may replace the assumptions
$\mu\in\RR^+(E_{0,T})$, $g(\cdot,\cdot,0)\cdot\mu\in\RR(E_{0,T})$
by the assumption $g\cdot\mu\in\RR(E_{0,T})$ (see \cite[Theorem
5.8]{K:JFA}).
\end{remark}

\begin{remark}
(i) By \cite[Proposition 3.4]{K:JFA}, the solution $u$ of Theorem
\ref{th3.1} is quasi-continuous.
\smallskip\\
(ii) Let the assumptions of  Theorem \ref{th3.1} hold, and
moreover,  $f(\cdot,\cdot,0)\in L^1(E_{0,T};m_1)$,
$g(\cdot,\cdot,0)\cdot\mu\in\MM_{0,b}(E_{0,T})$ and for  some
$\gamma\ge\alpha_0$ the form $\EE^{0,T}_\gamma$ has the dual
Markov property (for the definition of $\EE^{0,T}$ see
\cite[Section 3.3]{K:JFA}). Then by \cite[Proposition 3.13]{K:JFA}
and (\ref{eq3.33}),
\begin{align*}
\|f_{u}\|_{L^1(E_{0,T};m_1)} +\|g_{u}\cdot\mu\|_{TV} &\le
c(\|\varphi\|_{L^1(E;m)}+\|f(\cdot,\cdot,0)\|_{L^1(E_{0,T};m_1)}\\
&\quad+\|g(\cdot,\cdot,0)\cdot\mu\|_{TV}),
\end{align*}
where $\|\cdot\|_{TV}$ denotes the total variation norm. Therefore,
by \cite[Theorem 3.12]{K:JFA},  $u\in L^1(E_{0,T};m_1)$,
$T_ku\in L^2(0,T;V)$
for $k>0$ ($T_ku$ is
defined by (\ref{eq3.32})) and for every $k>0$ there is $C>0$
depending only on $k,\alpha,T$ such that
\[
\int^T_0B^{(t)}(T_k\bar u(t),T_k\bar u(t))\,dt\le
C(\|\varphi\|_{L^1(E;m)} +\|f(\cdot,0)\|_{L^1(E_{0,T};m_1)}
+\|g(\cdot,\cdot,0)\cdot\mu\|_{TV}).
\]
Moreover, if the forms $(B^{(t)},V)$ are (non-symmetric)
Dirichlet forms, then by \cite[Theorem 4.5]{KR:NoD}, $u$ is a
renormalized solution of (\ref{eq3.2}) in the sense defined in
\cite{KR:NoD}.
\end{remark}

\begin{remark}
In Theorem \ref{th3.1} we have assumed that the AF $A^{\mu}$ is
continuous. In the general case where $\mu\in\RR^+(E_{0,T})$ and
$A^{\mu}$ is possibly discontinuous,  one can prove the existence
of a solution of (\ref{eq3.2}) in the following sense: there
exists $u:E_T\rightarrow\BR$ such that $f_{u}\cdot m,
g_{u}\cdot\mu \in\RR(E_{0,T})$ and (\ref{eq4.1}) is satisfied with
$g_{u}$ replaced by $g_{\hat u}$, where $\hat u$ is the precise
version of $u$ (for the notion of a precise version of a parabolic
potential see \cite{Pi}). In the paper we decided to provide the
proof of less general result, because it suffices for the purposes
of Sections \ref{sec4}--\ref{sec6} in which our main results are
proved, and on the other hand, the proof of the general result is
more technical than the proof of Theorem \ref{th3.1}. Also note
that by \cite[Proposition 3.4]{K:JFA}, the solution $u$ described
above is quasi-c\`adl\`ag.
\end{remark}

\section{Convergence of BSDEs and Elliptic PDEs}
\label{sec4}

In this section, we assume that (\ref{eq5.12}). We denote by $L$
the operator associated via (\ref{eq2.14}) with the form $(B,V)$.
We also assume that $\mu\in\RR^+(E_{0,T})$ does not depend on time
and $f,g:E\times\BR\rightarrow\BR$, i.e. $f,g$ also do not depend
on time. For $v\in E\rightarrow\BR$, we set
\[
f_v(x):=f(x,v(x)),\quad g_v(x):=g(x,v(x)),\quad x\in E.
\]
To shorten notation, in what follows we denote $P_{0,x}$
by $P_x$, $E_{0,x}$ by $E_x$ and $\|\cdot\|_{(0,x);1}$ by $\|\cdot\|_{x,1}$.
Under the measure $P_x$,
\begin{equation}
\label{eq3.31} \BBX_t=(t,X_t),\quad t\ge0,\qquad
\zeta_{\tau}=T\wedge\zeta.
\end{equation}
and
\[
A^{\mu}_t=A^{0,\tilde\mu}_t,\quad t\ge0,
\]
where $\tilde\mu$ is determined by (\ref{eq2.6}).

In the rest of the paper we say that some property is satisfied
quasi-everywhere (q.e. for brevity) if the set of those $x\in E$
for which it does not hold is exceptional with respect to the form
$(B,V)$.

Let $\nu\in S(E)$. We will say that a Borel measurable
$F:E\rightarrow\BR$ is $\nu$-quasi-integrable ($F\in qL^1(E;\nu)$
in notation) if for every $T>0$, $P_x(\int^{\zeta\wedge T}_0
|F(X_t)|\,dA^{0,\nu}_t<\infty)=1$ for q.e. $x\in E$.

Note that in case $\nu=m$ the notion of quasi-integrability was
introduced in \cite[Section 2]{K:AMPA}. For a comparison of the
notion of $m$-integrability and the notion of quasi-integrability
in the analytic sense see \cite[Remark 2.3]{K:AMPA}.

In this section and Section \ref{sec5}, we  assume that the
data satisfy the following conditions.
\begin{enumerate}
\item[(E1)]$\varphi\in L^1(E;m)$, $\tilde\mu\in\RR^{+}(E)$.

\item[(E2)]$f(\cdot,y),g(\cdot,y)$
are measurable for every $y\in\BR$ and $f(x,\cdot),g(x,\cdot)$ are
continuous for every $x\in E$.

\item[(E3)]
$(f(x,y)-f(x,y'))(y-y')\le0$ for all $y,y'\in\BR$ and
$x\in E$.

\item[(E4)]$f(\cdot,0)\cdot m\in \RR(E)$ and
$f(\cdot,y)\in qL^{1}(E;m)$ for every $y\in\BR$.

\item[(E5)]$(g(x,y)-g(x,y'))(y-y')\le0$ for all $y,y'\in\BR$ and $x\in E$.

\item[(E6)]$g(\cdot,0)\cdot\tilde\mu\in\RR(E)$ and
$g(\cdot,y)\in qL^{1}(E;\tilde{\mu})$ for every $y\in\BR$.
\end{enumerate}

\begin{definition}
Let $x\in E$. We say that a pair $(Y^x,M^x)$ is a solution of the BSDE
\begin{align}
\label{eq4.11} Y^x_{t}= \int_{t\wedge\zeta}^{\zeta}
f(X_s,Y^x_s)\,ds +\int_{t\wedge\zeta}^{\zeta}
g(X_s,Y^x_s)\,dA^{\mu}_s -\int_{t\wedge\zeta}^{\zeta}dM^x_s, \quad
t\ge0,
\end{align}
on the space $(\Omega,\FF,P_x)$ if
\begin{enumerate}
\item[(a)]$Y^x$ is an $(\FF_t)$-progressively measurable process
of class D under $P_x$, $Y^x_{t\wedge\zeta}\rightarrow0$,
$P_x$-a.s. as $t\rightarrow\infty$ and $M^x$ is an $(\FF_t)$-local
martingale under $P_x$ such that $M^x_0=0$,

\item[(b)]For every $T>0$, $\int^{T}_0|f(X_t,Y^x_t)|\,dt<\infty$,
$\int^{T}_0|g(X_t,Y^x_t)|\,dA^{\mu}_t<\infty$, $P_x$-a.s., and
\[
Y^x_{t}= Y^x_{T\wedge\zeta}+\int_{t\wedge\zeta}^{T\wedge\zeta}
f(X_s,Y^x_s)\,ds +\int_{t\wedge\zeta}^{T\wedge\zeta}
g(X_s,Y^x_s)\,dA^{\mu}_s
-\int_{t\wedge\zeta}^{T\wedge\zeta}dM^x_s,\quad t\in[0,T].
\]
\end{enumerate}
\end{definition}

\begin{definition}
We say that $v:E\rightarrow\BR$ is a solution of problem
(\ref{eq1.2}) with $\lambda=0$ if $f_{v}\cdot m\in\RR(E)$,
$g_v\cdot\tilde\mu\in\RR(E)$ and for q.e. $x\in E$,
\begin{equation}
\label{eq3.7} v(x)=E_{x}\left(\int_{0}^{\zeta}f_v(X_t)\,dt
+\int_{0}^{\zeta}g_v(X_t)\,dA^{\mu}_t\right).
\end{equation}
\end{definition}

Suppose that for some $x\in E$ for every $n>0$ there exists a
solution $(Y^n,M^n)$ of the BSDE
\begin{align}
\label{eq4.24} Y^n_t&=\fch_{\{\zeta>n\}} \varphi(X_n)
+\int_{t\wedge\zeta}^{n\wedge\zeta} f(X_{s},Y^n_{s})\,ds\nonumber\\
&\quad+\int_{t\wedge\zeta}^{n\wedge\zeta} g(X_{s},Y^n_{s})\,dA^{\mu}_{s}
-\int_{t\wedge\zeta}^{n\wedge\zeta}dM^n_{s},\quad t\in [0,n], \quad
P_{x}\mbox{-a.s.}
\end{align}
on the probability space $(\Omega,\FF,P_x)$. The solutions may
depend on $x$ but for brevity, in our notation we drop the dependence
of $Y^n,M^n$ on $x$. In what follows by $\tilde Y^n,\tilde M^n$ we
denote the processes defined as
\begin{equation}
\label{eq5.02} \tilde Y^n_t=Y^n_t,\quad\tilde M^n_t=M^n_t,\quad
t<n,\qquad \tilde Y^n_t=0,\quad \tilde M^n_t=M^n_n,\quad t\ge n.
\end{equation}

\begin{proposition}
\label{prop4.1}  Assume that \mbox{\rm(E1)--(E6)} are satisfied.
For $0<n<m$, we  set $\delta Y=\tilde Y^m-\tilde Y^n$. Then for every
$x\in E$,
\begin{align}
\label{eq5.18} \|\delta Y\|_{x,1} &\le
E_x\Big(\fch_{\{\zeta>m\}}|\varphi(X_m)|
+\fch_{\{\zeta>n\}}|\varphi(X_n)|\nonumber\\
&\qquad\quad\left.+\int^{m\wedge\zeta}_{n\wedge\zeta} |f(X_t,0)|\,dt
+\int^{m\wedge\zeta}_{n\wedge\zeta} |g(X_t,0)|\,dA^{\mu}_t\right)
\end{align}
and
\begin{align}
\label{eq5.19} E_{x}\sup_{t\ge0}|\delta Y_t|^q
&\le\frac{1}{1-q}\Big(E_{x}(\fch_{\{\zeta>m\}}|\varphi(X_m)|
+\fch_{\{\zeta>n\}}|\varphi(X_n))|) \nonumber\\
&\qquad\qquad\left. +E_x\int^{m\wedge\zeta}_{n\wedge\zeta}
|f(X_t,0)|\,dt+E_x\int^{m\wedge\zeta}_{n\wedge\zeta}
|g(X_t,0)|\,dA^{\mu}_t\right)^q
\end{align}
for every $q\in(0,1)$. Moreover, for every $t\ge0$,
\begin{align}
\label{eq4.8} &E_x\int_0^{t\wedge\zeta}|f(X_{s},Y^n_{s})|\,ds
+E_x\int_0^{t\wedge\zeta}|g(X_{s},Y^n_{s})|\,dA^{\mu}_s\nonumber \\
&\qquad\le E_x\left(|Y^n_t| +2\int_0^{t\wedge\zeta}|f(X_{s},0)|\,ds
+2\int^{t\wedge\zeta}_0|g(X_s,0)|\,dA^{\mu}_{s}\right).
\end{align}
\end{proposition}
\begin{proof}
By (\ref{eq4.24}),
\begin{align}
\label{eq4.9} Y^n_t&=Y^n_0-\int_0^{t\wedge\zeta}
\{f(X_{s},Y^n_{s})\,ds +g(X_{s},Y^n_{s})\,dA^{\mu}_{s}\}
+\int_0^{t\wedge\zeta}dM^n_{s}\nonumber\\
&=Y^n_0-\int_0^t(\fch_{[0,n\wedge\zeta]}(s) f(X_{s},Y^n_{s})\,ds
+\fch_{[0,n\wedge\zeta]}(s)
g(X_{s},Y^n_{s})\,dA^{\mu}_{s})\nonumber\\
&\quad +\int_0^t\fch_{[0,n\wedge\zeta]}(s)\,dM^n_{s}, \quad
t\in[0,n],\,P_{x}\mbox{-a.s.}
\end{align}
From the above and the fact that the process $A^{\mu}$ is
continuous  it follows that the pair $(\tilde Y^n,\tilde M^n)$
defined by (\ref{eq5.02}) satisfies
\begin{align}
\label{eq5.3} \tilde
Y^n_t&=Y^n_0-\int_0^t(\fch_{[0,n\wedge\zeta]}(s) f(X_{s},\tilde
Y^n_{s})\,ds +\fch_{[0,n\wedge\zeta]}(s)g(X_{s},\tilde Y^n_{s})
\,dA^{\mu}_{s}) \nonumber\\
&\quad+\int^t_0dV^n_{s}
+\int_0^t\fch_{[0,n\wedge\zeta]}(s)\,d\tilde M^n_{s}, \quad t\ge0,
\end{align}
where
\[
V^n_t=0\mbox{ if }t<n,\quad V^n_t=-Y^n_n\mbox{ if }t\ge n.
\]
Let $\delta\tilde Y=\tilde Y^m-\tilde Y^n$. By (\ref{eq5.3}),
\[
\delta\tilde Y_t=\delta \tilde
Y_0+K_t+\int^t_0(\fch_{[0,m\wedge\zeta]}(s) \,d\tilde
M^m_{s}-\fch_{[0,n\wedge\zeta]}(s) \,d\tilde M^n_{s}),\quad
t\ge0,\,P_{x}\mbox{-a.s.},
\]
where
\begin{align*}
K_t&=-\int^t_0\fch_{[0,n\wedge\zeta]}(s) (f(X_{s},\tilde
Y^m_{s})-f(X_{s},\tilde Y^n_{s}))\,ds
-\int^t_0\fch_{(n\wedge\zeta,m\wedge\zeta]}(s)
f(X_{s},\tilde Y^m_{s})\,ds\\
&\quad-\int^t_0\fch_{[0,n\wedge\zeta]}(s) (g(X_{s},\tilde Y^m_{s})
-g(X_{s},\tilde Y^n_{s}))\,dA^{\mu}_{s}\\
&\quad-\int^t_0\fch_{(n\wedge\zeta,m\wedge\zeta]}(s)
g(X_{s},\tilde Y^m_{s})\,dA^{\mu}_{s} +\int^t_0d(V^m_{s}-V^n_{s}).
\end{align*}
By the Meyer-Tanaka formula (see \cite[p. 216]{Pr}), for $t<m$ we
have
\[
|\delta\tilde Y_m|-|\delta\tilde Y_t|\ge\int^m_t\mbox{sgn}
(\delta\tilde Y_{s-})\,d(\delta\tilde Y)_{s},
\]
where $\mbox{sgn}(x)=1$ if $x>0$ and $\mbox{sgn}(x)=-1$ if
$x\le0$. Therefore, for $t<m$,
\[
|\delta\tilde Y_t|=E_x(|\delta\tilde Y_t|\,|\FF_t)\le
E_x\left(|\delta\tilde Y_m| -\int^m_t\mbox{sgn}(\delta\tilde
Y_{s-})\,dK_{s}\,\Bigl|\,\FF_t\right).
\]
From this it follows that for $t\in[0,m]$,
\begin{align*}
|\delta\tilde Y_t|&\le E_{x}\left(|\delta\tilde Y_m|
+\int_t^m\fch_{[0,n\wedge\zeta]}(s)\mbox{sgn}(\delta\tilde
Y_{s})(f(X_{s},\tilde Y^m_{s})-f(X_{s},\tilde Y^n_{s}))\,ds\right.\\
&\qquad\qquad +\int_t^m\fch_{[0,n\wedge\zeta]}(s)
\mbox{sgn}(\delta\tilde Y_{s})(g(X_{s},\tilde Y^m_{s})
-g(X_{s},\tilde Y^n_{s}))\,dA^{\mu}_{s}\\
&\qquad\qquad+\int^m_t\fch_{(n\wedge\zeta,m\wedge\zeta]}(s)
\mbox{sgn}(\delta Y_{s})f(X_{s},\tilde Y^m_{s})\,ds\\
&\qquad\qquad\left.+\int^m_t\fch_{(n\wedge\zeta,m\wedge\zeta]}(s)
\mbox{sgn}(\delta\tilde Y_{s}) g(X_{s},\tilde
Y^m_{s})\,dA^{\mu}_{s} +|V^m_m|+|V^n_n| \,\Bigl|\,\FF_t\right).
\end{align*}
By (E3),
\[
\int_t^m\fch_{[0,n\wedge\zeta]}(s)\mbox{sgn}(\delta \tilde
Y_{s})(f(X_{s},\tilde Y^m_{s})-f(X_{s},\tilde Y^n_{s}))\,ds\le0,
\]
whereas by (E5) and the fact that $A^{\mu}$ is increasing,
\[
\int_t^m\fch_{[0,n\wedge\zeta]}(s)\mbox{sgn}(\delta\tilde Y_{s})
(g(X_{s},\tilde Y^m_{s})-g(X_{s},\tilde Y^n_{s}))\,dA^{\mu}_{s}
\le0.
\]
Furthermore, since $\tilde Y^n_t=0$ for $t\ge n$, it follows from
(E3) that
\begin{align*}
\int^m_t\fch_{(n\wedge\zeta,m\wedge\zeta]}(s) \mbox{sgn}(\delta
\tilde Y_{s})f(X_{s},\tilde Y^m_{s})\,ds &\le
\int^m_t\fch_{(n\wedge\zeta,m\wedge\zeta]}(s)
\mbox{sgn}(\delta\tilde Y_{s})f(X_{s},0)\,ds\\
&\le \int^{m\wedge\zeta}_{n\wedge\zeta} |f(X_{s},0)|\,ds.
\end{align*}
Similarly, by (E5),
\[
\int^m_t\fch_{(n\wedge\zeta,m\wedge\zeta]}(s) \mbox{sgn}(\delta
\tilde Y_{s}) g(X_{s},\tilde Y^m_{s})\,dA^{\mu}_{s}
\le\int^{m\wedge\zeta}_{n\wedge\zeta} |g(X_{s},0)|\,dA^{\mu}_{s}.
\]
Furthermore, $\delta\tilde Y_m=0$ and
\[
|V^m_m|+|V^n_n|= |Y^m_m|+|Y^n_n|=\fch_{\{\zeta>m\}}|\varphi(X_m)|
+\fch_{\{\zeta>n\}}|\varphi(X_n)|.
\]
Therefore, for $t\in[0,m]$ we have
\begin{align}
\label{eq5.32} |\delta\tilde Y_t|&\le
E_{x}\Bigl(\fch_{\{\zeta>m\}}|\varphi(X_m)|
+\fch_{\{\zeta>n\}}|\varphi(X_n)|\nonumber\\
&\qquad\quad\left.+\int^{m\wedge\zeta}_{n\wedge\zeta} |f(X_{s},0)|\,ds
+\int^{m\wedge\zeta}_{n\wedge\zeta}
|g(X_{s},0)|\,dA^{\mu}_{s}\,\Bigl|\,\FF_t\right)=: N_t.
\end{align}
This implies (\ref{eq5.18}). By \cite[Lemma 6.1]{BDHPS},
\[
E_{x}\sup_{0\le t\le m}|\delta\tilde Y_t|^q
\le(1-q)^{-1}(E_{x}N_m)^q,
\]
which shows (\ref{eq5.19}).
%\\
%K: \\
%Bradley, tom 1, dodatek: Niech $q\in(0,1)$, $q\le p$, $a_i\ge0$.
%Wtedy $(\sum_{i=1}^na_i^p)^{1/p}\le (\sum_{i=1}^na_i^q)^{1/q}$.
%Biorac $p=1$ otrzymujemy $(\sum_{i=1}^na_i)^q\le
%\sum_{i=1}^na_i^q)$.
%\\
%KK:\\
Finally, to prove (\ref{eq4.8}), we first observe that by the
Meyer-Tanaka formula,
\[
E_x|Y^n_t|-E_x|Y^n_0|\ge
E_x\int^t_0\mbox{sgn}(Y^n_{s-})\,dY^n_s.
\]
By the above inequality and (\ref{eq4.9}), for $t<n$ we have
\begin{align}
\label{eq5.09} &E_x|Y^n_t|-E_x|Y^n_0|\nonumber\\
&\qquad\ge-E_x\int^t_0 \fch_{\{[0,n\wedge\zeta]}(s)
\mbox{sgn}(Y^n_{s})\{f(X_{s},Y^n_{s})\,ds +
g(X_{s},Y^n_{s})\,dA^{\mu}_{s}\}.
\end{align}
On the other hand, for every $t\ge0$,
\begin{align*}
\int^t_0|g(X_{s},Y^n_{s})|\,dA^{\mu}_s
&\le\int^t_0|g(X_{s},Y^n_{s})-g(X_{s},0)|\,dA^{\mu}_s
+\int^t_0|g(X_{s},0)|\,dA^{\mu}_s \\
&=-\int^t_0{\mbox{sgn}}(Y^n_{s})
(g(X_{s},Y^n_{s})-g(X_{s},0))\,dA^{\mu}_s
+\int^t_0|g(X_{s},0)|\,dA^{\mu}_s \\
&\le -\int^t_0{\mbox{sgn}}(Y^n_{s}) g(X_{s},Y^n_{s})\,dA^{\mu}_s
+2\int^t_0|g(X_{s},0)|\,dA^{\mu}_s,
\end{align*}
and similarly,
\[
\int^t_0|f(X_{s},Y^n_{s})|\,ds \le -\int^t_0{\mbox{sgn}}(Y^n_{s})
f(X_{s},Y^n_{s})\,ds +2\int^t_0|f(X_{s},0)|\,ds,
\]
which when combined with (\ref{eq5.09}) proves (\ref{eq4.8}).
\end{proof}

\begin{proposition}
\label{prop4.2} Assume that \mbox{\rm(E1)--(E6)} are satisfied and
\begin{equation}
\label{eq4.10}
\lim_{t\rightarrow\infty}E_x\fch_{\{\zeta>t\}}|\varphi(X_t)|=0.
\end{equation}
Assume also for some $x\in E$ for each  $n\in\BN$ there exists a
solution $(Y^n,M^n)$ of \mbox{\rm(\ref{eq4.24})} on the space
$(\Omega,\FF,P_x)$. If
\begin{equation}
\label{eq5.10} E_x\int^{\zeta}_{0}|f(X_{t},0)|\,dt
+E_x\int^{\zeta}_{0} |g(X_{t},0)|\,dA^{\mu}_{t}<\infty,
\end{equation}
then there exists a solution $(Y^x,M^x)$ of
\mbox{\rm(\ref{eq4.11})}  on $(\Omega,\FF,P_x)$. Moreover,
$Y^x\in\DD^q(P_x)$ for $q\in(0,1)$, $M^x$ is a uniformly
integrable $(\FF_t)$-martingale under $P_x$ and
\begin{align}
\label{eq5.28} &E_x\int^{\zeta}_0|f(X_t,Y^x_t)|\,dt
+E_x\int^{\zeta}_0|g(X_t,Y^x_t)|\,dA^{\mu}_t \nonumber\\
&\qquad\le 2E_x\left(\int^{\zeta}_0|f(X_t,0)|\,dt
+\int^{\zeta}_0|g(X_t,0)|\,dA^{\mu}_t\right).
\end{align}
Finally,
\begin{equation} \label{eq5.20}
\lim_{n\rightarrow\infty}\|Y^n-Y^x\|_{x,1}=0
\end{equation}
and for every $q\in(0,1)$,
\begin{equation}
\label{eq5.11}
\lim_{n\rightarrow\infty}E_x\sup_{t\ge0}|Y^n_t-Y^x_t|^q=0.
\end{equation}
\end{proposition}
\begin{proof}
From (\ref{eq5.18}) and (\ref{eq4.10}), (\ref{eq5.10}) it follows
that for every $x\in E$, $ \|Y^n-Y^m\|_{x,1}\rightarrow0$ as
$n,m\rightarrow\infty$. Hence there exists a process
$Y\in\LL^1(P_x)$ of class D such that (\ref{eq5.20}) is satisfied.
By (\ref{eq5.19}), (\ref{eq4.10}) and (\ref{eq5.10}),
$\lim_{n,m\rightarrow\infty}E_x\sup_{t\ge0}|Y^n_t-Y^m_t|^q
\rightarrow0$.  Since the space $\DD^q(P_x)$ is complete, the last
convergence and (\ref{eq5.20}) imply that $Y^x\in\DD^q(P_x)$ and
(\ref{eq5.11}) is satisfied. From (\ref{eq4.8}), (\ref{eq5.20}),
(\ref{eq5.11}) and Fatou's lemma it follows that for every $T>0$,
\begin{align*}
& E_x\int_0^{T\wedge\zeta}|f(X_{t},Y^x_{t})|\,dt
+E_x\int^{T\wedge\zeta}_0|g(X_t,Y^x_t)|\,dA^{\mu}_t\\
&\qquad\le 2E_x\left(|Y^x_{T\wedge\zeta}|
+\int_0^{T\wedge\zeta}|f(X_{t},0)|\,ds +\int^{T\wedge\zeta}_0
|g(X_t,0)|\,dA^{\mu}_{t}\right).
\end{align*}
Since $\fch_{\{n\ge\zeta\}}Y^n_{\zeta}=0$ $P_x$-a.s. for $n\in\BN$, from
(\ref{eq5.11}) we conclude that $Y^x_{T\wedge\zeta}\rightarrow0$
in probability $P_x$ as $T\rightarrow\infty$. As a consequence,
since $Y^x$ is of class D, $E_x|Y^x_{T\wedge\zeta}|\rightarrow0$.
Therefore letting $T\rightarrow\infty$ in the last inequality we
get (\ref{eq5.28}). Using  (\ref{eq5.11}) one can show that
$\int^{\zeta}_0 |g(X_{t},Y^n_{t})-g(X_{t},Y^x_{t})|
\,dA^{\mu}_{t}\rightarrow0$ in probability $P_x$ (see the proof of
\cite[(6.16)]{K:JEE}).
%\\
%K: \\
%Dla ustalonych $\varepsilon,R>0$,
%\begin{align*}
%&P_x\Big(\int^{\zeta^{\lambda}}_0|g(X_{s},Y^n_{s})
%-g(X_{s},Y_{s})|\,ds>\varepsilon\Big)\\
%&\quad\le P_x\Big(\int^{\zeta^{\lambda}}_0|g(\BBX_{s},Y^n_{s})
%-g(\BBX_{s},Y_{s})|\,ds>\varepsilon,
%\sup_{t\ge0}|Y_t|\le R,\sup_{t\ge0}|Y^n_t-Y_t|\le R\Big)\\
%&\qquad +P_x(\sup_{t\ge0}|Y_t|>R) +P_x(\sup_{t\ge0}|Y^n_t-Y_t|>R).
%\end{align*}
%KK\\
Set $F_R(t,x)=|f(t,x,-R)|\vee|f(t,x,R)|$,
$G_R(t,x)=|g(t,x,-R)|\vee|g(t,x,R)|$ and for $N,R>0$ and $n\in\BN$
define the stoping times
\[
\tau_{n,R}=\inf\{t\ge0:|Y^n_t|>R\},\qquad \tau_R=\inf_{n\ge R}
\tau_{n,R}
\]
and
\[
\sigma_{N,R}=\inf\left\{t\ge0:\int^t_0(F_R(X_s)\,ds+G_R(X_s)\,dA^{\mu}_s)>N\right\},
\qquad \delta_{N,R}=\sigma_{N,R}\wedge\tau_R.
\]
By (\ref{eq4.24}), for $T<n$ we have
%dla  $t<n$ mamy
%\begin{align*}
%Y^n_{t\wedge\zeta^{\lambda}\wedge\sigma_R}&=\fch_{\{\zeta^{\lambda}>n\}}
%\varphi(X_n)
%+\int_{t\wedge\zeta^{\lambda}\wedge\sigma_R}^{n\wedge\zeta^{\lambda}}
%f(X_{s},Y^n_{s})\,ds\\
%&\quad+\int_{t\wedge\zeta^{\lambda}\wedge\sigma_R}^{n\wedge\zeta^{\lambda}}
%g(X_{s},Y^n_{s})\,dA^{\lambda,\mu}_{s}
%-\int_{t\wedge\zeta^{\lambda}\wedge\sigma_R}^{n\wedge\zeta^{\lambda}}dM^n_{s},
%\quad P_{x}\mbox{-a.s.} \nonumber
%\end{align*}
\begin{align*}
Y^n_{t\wedge\zeta\wedge\delta_{N,R}}
&=Y^n_{T\wedge\zeta\wedge\delta_{N,R}}
+\int^{T\wedge\zeta\wedge\delta_{N,R}}
_{t\wedge\zeta\wedge\delta_{N,R}} \{f(X_{s},Y^n_{s})\,ds
+g(X_{s},Y^n_{s})\,dA^{\mu}_{s}\}\\
&\quad-\int^{T\wedge\zeta\wedge\delta_{N,R}}
_{t\wedge\zeta\wedge\delta_{N,R}}dM^n_{s},\quad t\in[0,T],\quad
P_{x}\mbox{-a.s.}
\end{align*}
Since $Y^n_t=Y^n_{t\wedge\zeta}$ and
$\int^{T\wedge\zeta\wedge\delta_{N,R}}
_{t\wedge\zeta\wedge\delta_{N,R}}dM^n_{s}
=\int^T_tdM^n_{s\wedge\zeta\wedge\delta_{N,R}}$ and the martingale
$M^n$ stopped at $\zeta\wedge\delta_{N,R}$ is still a martingale
(see \cite[Theorem I.18]{Pr}), it follows that
\begin{equation}
\label{eq5.21} Y^n_{t\wedge\zeta\wedge\delta_{N,R}}
=E_x\left(Y^n_{T\wedge\zeta\wedge\delta_{N,R}}
+\int^{T\wedge\zeta\wedge\delta_{N,R}}
_{t\wedge\zeta\wedge\delta_{N,R}} \{f(X_{s},Y^n_{s})\,ds
+g(X_{s},Y^n_{s})\,dA^{\mu}_{s}\}\,\Bigl|\,\FF_t\right).
\end{equation}
By Doob's inequality (see, e.g., \cite[Theorem 1.9.1]{LS}) and
(\ref{eq5.20}), for every $\varepsilon>0$ we have
\begin{align}
\label{eq5.24} &\lim_{n\rightarrow\infty}P_x(\sup_{t\le T}
|E_x(Y^n_{T\wedge\zeta\wedge\delta_{N,R}}
-Y^x_{T\wedge\zeta\wedge\delta_{N,R}}\,|\FF_t\,)|>\varepsilon) \nonumber\\
&\qquad\le \varepsilon^{-1}\lim_{n\rightarrow\infty}
E_x|Y^n_{T\wedge\zeta\wedge\delta_{N,R}}
-Y^x_{T\wedge\zeta\wedge\delta_{N,R}}|=0.
\end{align}
From the definition of  $\delta_{N,R}$ and (E2), (\ref{eq5.11}) it
follows that
\[
\lim_{n\rightarrow\infty}E_x\int^{T\wedge\zeta\wedge\delta_{N,R}}_0
\{|f(X_{s},Y^n_{s})-f(X_{s},Y^x_{s})|\,ds
+|g(X_{s},Y^n_{s})-g(X_{s},Y^x_{s})|\,dA^{\mu}_s\}=0.
\]
%K: Powy/zsza granica jest r/owna
%\[
%\lim_{n\rightarrow\infty}E_x\int^{T\wedge\zeta\wedge\delta_{N,R}-}_0
%(|f(X_{s},Y^n_{s})-f(X_{s},Y^x_{s})|\,ds
%+|g(X_{s},Y^n_{s})-g(X_{s},Y^x_{s})|\,dA^{\mu}_s)=0.
%\]
%KK
Hence, by Doob's inequality (see, e.g., \cite[Theorem 1.9.1]{LS}),
\begin{equation}
\label{eq5.29} \lim_{n\rightarrow\infty}P_x\left(\sup_{t\le T}
\left|E_x\left(\int^{T\wedge\zeta\wedge\delta_{N,R}}_
{t\wedge\zeta\wedge\delta_{N,R}}
(f(X_{s},Y^n_{s})-f(X_{s},Y^x_{s}))\,ds\,\Bigl|\,\FF_t\right)\right|
>\varepsilon\right)=0
\end{equation}
and
\begin{equation}
\label{eq5.30} \lim_{n\rightarrow\infty}P_x\left(\sup_{t\le T}
\left|E_x\left(\int^{T\wedge\zeta\wedge\delta_{N,R}}_
{t\wedge\zeta\wedge\delta_{N,R}} (g(X_{s},Y^n_{s})
-g(X_{s},Y^x_{s}))\,dA^{\mu}_{s}\,\Bigl|\,\FF_t\right)\right|
>\varepsilon\right)=0
\end{equation}
for every $\varepsilon>0$. Letting $n\rightarrow\infty$ in
(\ref{eq5.21})  and using (\ref{eq5.11}) and
(\ref{eq5.24})--(\ref{eq5.30}) we conclude that $P_x$-a.s.
\begin{equation}
\label{eq2.25} Y^x_{t\wedge\delta_{N,R}}
=E_x\left(Y^x_{T\wedge\zeta\wedge\delta_{N,R}}
+\int^{T\wedge\zeta\wedge\delta_{N,R}}_{t\wedge\zeta\wedge\delta_{N,R}}
\{f(X_{s},Y^x_{s})\,ds+g(X_{s},Y^x_{s})\,dA^{\mu}_{s}\} \,\Bigl|\,\FF_t\right)
\end{equation}
for $t\in[0,T]$. By (E4), $F_R\in qL^1(E_{0,T};m_1)$, and by (E6),
$G_R\in qL^1(E_{0,T};\mu)$. Therefore
$\sigma_{N,R}\nearrow\tau_R$ $P_x$-a.s. as $N\rightarrow\infty$
for each fixed $R>0$. Hence
$Y^x_{T\wedge\zeta\wedge\delta_{N,R}}\rightarrow
Y^x_{T\wedge\zeta\wedge\tau_R}$ $P_x$-a.s. as
$N\rightarrow\infty$, and consequently
$E_x|Y^x_{T\wedge\zeta\wedge\delta_{N,R}}-
Y^x_{T\wedge\zeta\wedge\tau_R}|\rightarrow0$ since  $Y^x$ is of
class D. From the last convergence and Doob's inequality it
follows that for every $\varepsilon>0$,
\[
\lim_{N\rightarrow\infty}P_x(\sup_{t\le T}
|E_x(Y^x_{T\wedge\zeta\wedge\delta_{N,R}}-Y^x_{T\wedge\zeta\wedge\tau_R}
|\FF_t)|>\varepsilon)=0.
\]
Therefore letting $N\rightarrow\infty$ in (\ref{eq2.25}) and using
(\ref{eq5.28})  we show that $P_x$-a.s.,
\begin{equation}
\label{eq5.26} Y^x_t=E_x\left(Y^x_{T\wedge\zeta\wedge\tau_R}
+\int^{T\wedge\zeta\wedge\tau_R}_{t\wedge\zeta\wedge\tau_R}
\{f(X_{s},Y^x_{s})\,ds+g(X_{s},Y^x_{s})\,dA^{\mu}_{s}\}\,\Bigl|\,\FF_t\right),
\quad t\in[0,T].
\end{equation}
We now show that $\tau_R\nearrow\infty$ $P_x$-a.s. as
$R\rightarrow\infty$. To see this, let us suppose that
$P_x(\sup_{R>0}\tau_R\le M)>\varepsilon$ for some
$M,\varepsilon>0$. Then
\begin{equation}
\label{eq4.25} P_x(\forall_{R>0}\,\,\, \sup_{n\ge R}\sup_{t\le
M}|Y^n_t|\ge R)>\varepsilon.
\end{equation}
Clearly,
\begin{align}
\label{eq4.26} P_x(\forall_{R>0}\,\,\, \sup_{n\ge R}\sup_{t\le M}
|Y^n_t|\ge R) &\le P_x(\forall_{R>0}\,\,\, \sup_{n\ge R}\sup_{t\le
M}|Y^n_t-Y_t|\ge
R/2)\nonumber\\
&\quad + P_x(\forall_{R>0}\,\,\,\sup_{t\le M}|Y_t|\ge R/2)\nonumber\\
&=P(\forall_{R>0}\,\,\, \sup_{n\ge R}\sup_{t\le M}|Y^n_t-Y_t|\ge
R/2).
\end{align}
By (\ref{eq5.11}), taking a subsequence if necessary, we may
assume that $\sup_{t\le M}|Y^n_t-Y_t|\rightarrow 0$ $P_x$-a.s.
Therefore the random  variable $Z=\sup_{n\ge 0}\sup_{t\le
M}|Y^n_t-Y_t|$ is finite a.s., which when combined with
(\ref{eq4.26}) contradicts (\ref{eq4.25}). This proves that
$\tau_R\nearrow\infty$ $P_x$-a.s. Now, letting
$R\rightarrow\infty$ and repeating argument used to prove
(\ref{eq5.26}), we get (\ref{eq5.26}) with
$T\wedge\zeta\wedge\tau_R$ replaced by $T\wedge\zeta$. Since we
know that $E_x|Y^x_{T\wedge\zeta}|\rightarrow0$ as
$T\rightarrow\infty$, letting $T\rightarrow\infty$ in this
equation (i.e. in (\ref{eq5.26}) with $T\wedge\zeta$) and
repeating once again the argument used to prove (\ref{eq5.26}) we
get
\begin{equation}
\label{eq5.17} Y^x_t=E_x\left(\int^{\zeta}_{t\wedge\zeta}
\{f(X_{s},Y^x_{s})\,ds +g(X_{s},Y^x_{s})\,dA^{\mu}_{s}\}
\,\Bigl|\,\FF_t\right),\quad t\ge0,\quad P_x\mbox{-a.s.}
\end{equation}
%K:\\
%Dostajemy powy/zsza r/owno/s/c dla $t\in T]$ dla ka/zdego
%ustalonego $T>0$, wi/ec dla $t\ge0$.\\
%KK\\
Hence
\begin{equation}
\label{eq5.23} Y^x_t=\int^{\zeta}_{t\wedge\zeta}
f(X_{s},Y^x_{s})\,ds +\int^{\zeta}_{t\wedge\zeta}
g(X_{s},Y^x_{s})\,dA^{\mu}_{s}
-\int^{\zeta}_{t\wedge\zeta}dM^x_{s},\quad t\ge0,\quad
P_x\mbox{-a.s.},
\end{equation}
where $M^x$ is a c\`adl\`ag version of the martingale
\begin{equation}
\label{eq4.31} t\mapsto E_x\left(\int^{\zeta}_0
f(X_{s},Y^x_{s})\,ds +\int^{\zeta}_0
g(X_{s},Y^x_{s})\,dA^{\mu}_{s}\,\Bigl|\,\FF_t\right)-Y^x_0.
\end{equation}
Indeed, by (\ref{eq5.17}),
\begin{align*}
Y^x_t&=E_x\left(\int^{\zeta}_0 f(X_{s},Y^x_{s})\,ds +\int^{\zeta}_0
g(X_{s},Y^x_{s})\,dA^{\mu}_{s}\,\Bigl|\,\FF_t\right)\\
&\quad-\int^{t\wedge\zeta}_0f(X_{s},Y^x_{s})\,ds
-\int^{t\wedge\zeta}_0 g(X_{s},Y^x_{s})\,dA^{\mu}_{s},\quad t\ge0,
\end{align*}
that is
\[
Y^x_t=Y^x_0+M^x_t-\int^{t\wedge\zeta}_0
f(X_{s},Y^x_{s})\,ds-\int^{t\wedge\zeta}_0
g(X_{s},Y^x_{s})\,dA^{\mu}_{s},\quad t\ge0.
\]
From the above it follows that $M^x_{t\wedge\zeta}=M^x_t$,
$t\ge0$, and moreover, that
\[
Y^x_t=Y_{T\wedge\zeta}
+\int^{T\wedge\zeta}_{t\wedge\zeta}f(X_s,Y^x_s)\,ds
+\int^{T\wedge\zeta}_{t\wedge\zeta} g(X_s,Y^x_s)\,dA^{\mu}_s
-\int^{T\wedge\zeta}_{t\wedge\zeta}dM^x_s,\quad t\ge0.
\]
Letting $T\rightarrow\infty$ and using the fact that
$Y^x_{T\wedge\zeta}\rightarrow Y^x_{\zeta}=0$ $P_x$-a.s. we
obtain (\ref{eq5.23}). Thus the pair $(Y^x,M^x)$ is a solution of
(\ref{eq4.11}).
\end{proof}

\begin{theorem}
\label{th4.3} Assume \mbox{\rm(\ref{eq5.12})} and assume that
$f,g,\mu$ do not depend on time and satisfy \mbox{\rm(E1)--(E6)}.
\begin{enumerate}
\item[\rm(i)]There exists a unique solution $v$ of problem
\mbox{\rm(\ref{eq1.2})} with $\lambda=0$.
\item[\rm(ii)]
Let
\[
M^x_t=E_x\left(\int_{0}^{\zeta} f_v(X_r)\,dr +\int_{0}^{\zeta}
g_v(X_r)\,dA^{\mu}_r\,\Bigl|\,\FF_{t}\right)-v(X_0),\quad t\ge0.
\]
Then there is a c\`adl\`ag $(\FF_t)$-adapted process $M$ such that
$M_t=M^x_t$, $t\ge0$, $P_x$-a.s. for q.e $x\in E$ and for q.e.
$x\in E$ the pair $(v(X),M)$ is a unique solution of
\mbox{\rm(\ref{eq4.11})} on the space $(\Omega,\FF,P_x)$.
Moreover,  $v(X)\in\DD^q(P_x)$ for $q\in(0,1)$ and $M$ is a
uniformly integrable martingale under $P_x$ for q.e. $x\in E$.
\end{enumerate}
\end{theorem}
\begin{proof}
We  first prove part (ii). The uniqueness of a solution of
(\ref{eq4.11}) follows easily from (E3), (E5) and the fact that
$\mu$ is positive. To see this it suffices to modify slightly the
proof of \cite[Proposition 3.1]{KR:JFA}.  To prove the existence
of a solution, we first note that by Theorem \ref{th3.1}, for q.e.
$x\in E$ for every $n\in\BN$ there exists a unique solution
$(Y^{n},M^{n})$ of the BSDE (\ref{eq4.24}) with $\varphi\equiv0$
on the space $(\Omega,\FF,P_x)$. Since $f(\cdot,0)\cdot
m,g(\cdot,0)\cdot\tilde\mu\in\RR(E)$, condition (\ref{eq5.10})  is
satisfied for q.e. $x\in E$. Therefore, by Proposition
\ref{prop4.2}, for q.e. $x\in E$ there exist a solution
$(Y^x,\tilde M^x)$ of BSDE (\ref{eq4.11}). In fact, $Y^x$ is given
by (\ref{eq5.17}) and $\tilde M^x$ is a c\`adl\`ag version of the
martingale (\ref{eq4.31}). Repeating step by step the proof of
\cite[Theorem 4.7]{KR:JFA} one can show that there is a pair of
c\`adl\`ag processes $(Y,M)$ not depending on $x$ such that
$(Y_t,M_t)=(Y^x_t,\tilde M^x_t)$, $t\ge0$, $P_x$-a.s. for q.e.
$x\in E$, and secondly, that in fact $Y=v(X)$, where
$v(x)=E_xY_0$. This shows that the pair $(v(X),M)$ is a solution
of (\ref{eq4.11}) on the space $(\Omega,\FF,P_x)$ for q.e. $x\in
E$. By Proposition \ref{prop4.2}, $v(X)\in\DD^q(P_x)$ for
$q\in(0,1)$, and $M$ is a uniformly integrable $(\FF_t)$-martingale
under $P_x$. This completes the proof of (ii). Part (i) follows
immediately from (ii), because $g_v\cdot\mu\in\RR(E)$ and
(\ref{eq5.28}) is satisfied with $Y^x$ replaced by $v(X)$, so for
q.e. $x\in E$ we can integrate with respect to $P_x$ both sides of
(\ref{eq4.11}) with $t=0$ and $Y^x$ replaced by $v(X)$.
\end{proof}

\begin{remark}
If $g$ does not depend on the last variable $y$, then in Theorem
\ref{th4.3} we may replace the assumptions $\tilde\mu\in\RR^+(E)$,
$g(\cdot,0)\cdot\tilde\mu\in\RR(E)$ by the assumption
$g\cdot\tilde\mu\in\RR(E)$ (see \cite[Theorem 3.8]{KR:CM}).
\end{remark}

\begin{remark}
(i) By \cite[Lemma 4.3]{KR:JFA}, the solution $v$ of (\ref{eq1.2})
appearing in Theorem \ref{th4.3} is quasi-continuous. \smallskip\\
(ii) In addition to the hypotheses of Theorem \ref{th4.3} let us
assume that $(B,V)$ is a transient Dirichlet form and $f(\cdot,0)\in
L^1(E,m)$, $g(\cdot,0)\cdot\tilde\mu\in\MM_b(E)$, where
$\tilde\mu$ is determined by (\ref{eq2.6}). Then by
(\ref{eq5.28}), Lemma \ref{lem2.1}, the fact that $Y^x_t=v(X_t)$,
$t\ge0$, $P_x$-a.s. and \cite[Lemma 2.6]{KR:CM} (see also
\cite[Lemma 5.4] {KR:JFA}),
\[
\|f_v\|_{L^1(E;m)}+\|g_v\cdot\tilde\mu\|_{TV}\le
\|f(\cdot,0)\|_{L^1(E;m)}+\|g(0,\cdot)\cdot\tilde\mu\|_{TV}.
\]
Therefore, by  \cite[Theorem 4.2]{KR:CM} (see also
\cite[Proposition 5.9]{KR:JFA}), $f_v\in L^1(E;m)$, $T_kv$ belongs
to the extended Dirichlet space $V_e$ and for every $k>0$,
\[
B(T_kv,T_kv)\le k(\|f(\cdot,0)\|_{L^1(E;m)}
+\|g(0,\cdot)\cdot\tilde\mu\|_{TV}).
\]
Moreover, if $(B,V)$ is a (non-symmetric) Dirichlet form
satisfying the strong sector condition, then by \cite[Theorem
3.5]{KR:NoD}, $v$ is a renormalized solution of problem
(\ref{eq1.2}) in the sense defined in \cite{KR:NoD}.
\end{remark}

\begin{remark}
\label{rem4.5} If a family $\{B^{(t)},t\in\BR\}$ satisfies the
assumptions of Section \ref{sec2}, then for every $\lambda>0$ the
family $\{B^{(t)}_{\lambda},t\in\BR\}$, where
$B^{(t)}_{\lambda}(\varphi,\psi)
=B^{(t)}(\varphi,\psi)+\lambda(\varphi,\psi)_H$,
satisfies these  assumptions as well. Therefore all the results of
Sections \ref{sec3} and \ref{sec4} apply to the operators
associated with $B^{(t)}_{\lambda}$ and to the Markov process
associated with the form $\EE_{\lambda}$ defined by
(\ref{eq2.23}), (\ref{eq2.24}) but with $B^{(t)}$ replaced by
$B^{(t)}_{\lambda}$.
\end{remark}

\section{Large time asymptotics}
\label{sec5}

In this section, as in Section \ref{sec4}, we  assume that
(\ref{eq5.12}) is satisfied  and the data $f,g,\mu$ do not depend
on time. We denote by $L$  the operator corresponding to $(B,V)$.
We continue to write $P_{x}$ for $P_{0,x}$ and $E_{x}$ for
$E_{0,x}$,  and as in Section \ref{sec4}, the abbreviation ``q.e."
means quasi-everywhere with respect to the capacity determined by
$(B,V)$.

Suppose that for every $T>0$ there  exists a unique solution $u_T$
of (\ref{eq3.2}) with $L$ and the data $f,g,\mu$ satisfying the
above assumptions. By Remark \ref{rem3.1}, by putting
\[
u(t,x)=\bar u_T(t,x)=u_T(T-t,x),\quad t\in[0,T],\, x\in E,
\]
we define a probabilistic solution $u$ of (\ref{eq1.1}), i.e.
solution of the problem
\begin{equation}
\label{eq5.37} \left\{\begin{array}{l}
\partial_tu-Lu=f(x,u)+g(x,u)\cdot\mu
\quad\mbox{in }(0,\infty)\times E,
\medskip\\
u(0,\cdot)=\varphi\quad\mbox{on }E.
\end{array}
\right.
\end{equation}
Our goal is to prove that under suitable assumptions,
$u(t,x)\rightarrow v(x)$ as $t\rightarrow\infty$ for q.e. $x\in
E$, where $v$ is a solution of (\ref{eq1.2}) with $\lambda=0$, i.e. solution of the
problem
\begin{equation}
\label{eq5.38} -Lv=f(x,v)+g(x,v)\cdot\tilde\mu\quad \mbox{in } E,
\end{equation}
where $\tilde\mu$ is determined by (\ref{eq2.6}). We will also estimate the rate of the convergence. The proofs of
these results rely on the results of Section \ref{sec4}. The main
idea is as follows. We have
\begin{equation}
\label{eq5.1} u(t,x)=u_T(T-t,x),\quad t\in[0,T],\, x\in E,
\end{equation}
where $u_T$ is a solution of the problem
\begin{equation}
\label{eq5.39} \partial_tu_T+Lu_T=-f(x,u_T) -g(x,u_T),\qquad
u_T(T)=\varphi.
\end{equation}
In particular, putting $t=T$, we get $u(T,x)=u_T(0,x)$. Hence, by
(\ref{eq4.1}),
\begin{equation}
\label{eq5.4} u(T,x)=E_{x}\left(\varphi(\BBX_{T\wedge\zeta})
+\int_{0}^{T\wedge\zeta}f_{u_T} (\BBX_{t})\,dt
+\int_{0}^{T\wedge\zeta}g_{u_T}(\BBX_{t})\,dA_{t}^{\mu}\right),
\end{equation}
because $\zeta_{\tau}=T\wedge\zeta$ under the measure $P_x$. On
the other hand, by Lemma \ref{lem2.1},
\begin{equation}
\label{eq5.41} v(x)=E_{x}\left(\int_{0}^{\zeta}f_v(X_t)\,dt
+\int_{0}^{\zeta}g_v(X_t)\,dA^{0,\tilde\mu}_t\right).
\end{equation}
Therefore our problem reduces to showing that the right-hand side
of (\ref{eq5.4}) converges to the right-hand side of
(\ref{eq5.41}) as $T\rightarrow\infty$, and to estimating the
difference between the two expressions  by some function of $T$.

In what follows, we denote by $(P_t)_{t\ge0}$, $(R_{\alpha})_{\alpha>0}$
the semigroup and the resolvent associated with the process
$\BM^{(0)}=(X,P_x)$ with life time $\zeta^0=\zeta$ (see Section
\ref{sec2.2}), i.e.
\[
P_tf(x)=E_xf(X_t),\qquad
R_{\alpha}f(x)=E_x\int^{\infty}_0e^{-\alpha t}f(X_t)\,dt,\quad
x\in E,f\in\BB_b(E).
\]
For $\nu\in \RR(E)$, we  set
\[
R_{\alpha}\nu(x)=E_x\int^{\zeta}_0e^{-\alpha t}\,dA^{0,\nu}_t
=E_x\int^{\infty}_0e^{-\alpha t}\,dA^{0,\nu}_t,
\]
where $A^{0,\nu}$ is the continuous AF of $\BM^{(0)}$ associated
with $\nu$ in the Revuz sense. Note that if $(B,V)$ is transient,
then $R_{\alpha}\nu$ is defined for $\alpha=0$.

Before stating our main result, let us note that with the
convention made at the beginning of Section \ref{sec2.2},
$E_x\fch_{\{\zeta>t\}}\psi(X_t)=P_t\psi(x)$ for Borel measurable $\psi\in
L^1(E;m)$, $t\ge0$. Therefore (\ref{eq4.10}) is equivalent to
\begin{equation}
\label{eq5.45} \lim_{t\rightarrow\infty}P_t|\varphi|(x)=0.
\end{equation}
Clearly, assumption (\ref{eq5.10}) is equivalent to
\begin{equation}
\label{eq5.46} R_0|f(\cdot,0)|(x)
+R_0(|g(\cdot,0)|\cdot\tilde\mu)(x)<\infty.
\end{equation}
By remarks given in Section \ref{sec2.2}, if $f(\cdot,0)\cdot
m\in\RR(E)$ and $g(\cdot,0)\cdot\tilde\mu\in\RR(E)$, then
(\ref{eq5.46}) is satisfied for q.e. $x\in E$.

\begin{theorem}
\label{th5.2}  Assume that  the assumptions of Theorem \ref{th4.3}
hold,  and moreover, \mbox{\rm(\ref{eq5.45})} is satisfied. Let
$u$ be a solution of \mbox{\rm(\ref{eq5.37})} and $v$ be a
solution of \mbox{\rm(\ref{eq5.38})}. Then
\begin{equation}
\label{eq5.14} \lim_{T\rightarrow\infty} u(T,x)=v(x)
\end{equation}
for q.e. $x\in E$. In fact,  for q.e. $x\in E$,
\begin{align}
\label{eq5.5} |u(T,x)-v(x)|\le
3P_T|\varphi|(x)+3P_T(R_0(|f(\cdot,0)|
+|g(\cdot,0)|\cdot\tilde\mu))(x)
\end{align}
for all $T>0$.
\end{theorem}
\begin{proof}
Let $Y^T$ be the first component of the solution of (\ref{eq4.24})
(with $T=n$) and $Y$ be the first component of the solution of
(\ref{eq4.11}). Since (\ref{eq5.10}) is satisfied for q.e. $x\in
E$, applying Proposition \ref{prop4.2} we conclude that for every
$q\in(0,1)$,
\begin{equation}
\label{eq5.6} \lim_{T\rightarrow\infty}E_x|Y^T_0-Y_0|^q=0
\end{equation}
for q.e. $x\in E$. On the other hand, by Theorem \ref{th3.1} and
Theorem \ref{th4.3}, for q.e. $x\in E$ we have
\[
Y^T_t=u_T(\BBX_t),\qquad Y_t=v(X_t),\quad t\ge0,\quad
P_x\mbox{-a.s.},
\]
where $u_T$ is a solution of (\ref{eq5.39}) and $v$ is a solution
of (\ref{eq3.7}). In particular, for q.e. $x\in E$,
\[
Y^T_0=u_T(0,x),\qquad Y_0=v(x),\quad P_x\mbox{-a.s.}
\]
But $u_T(0,x)=u(T,x)$ by (\ref{eq5.1}). Hence
\begin{equation}
\label{eq5.13} |u(T,x)-v(x)|^q=|u_T(0,x)-v(x)|^q=E_x|Y^T_0-Y_0|^q
\end{equation}
for $T>0$. Therefore (\ref{eq5.6}) implies (\ref{eq5.14}). To show
(\ref{eq5.5}), we first observe that by (\ref{eq5.6}) and
(\ref{eq5.13}),
\begin{equation}
\label{eq5.42} |u(T,x)-v(x)|^q
=\lim_{m\rightarrow\infty}E_x|Y^T_0-Y^m_0|^q,
\end{equation}
whereas by (\ref{eq5.19}) and (\ref{eq4.10}),
\begin{align}
\label{eq5.40} \lim_{m\rightarrow\infty}E_x|Y^T_0-Y^m_0|^q &\le
\frac{1}{1-q}\left(E_{x} \fch_{\{\zeta>T\}}|\varphi(X_T))|
+\int^{\zeta}_{T\wedge\zeta}|f(X_t,0)|\,dt\right.\nonumber\\
&\qquad\qquad\left. +\int^{\zeta}_{T\wedge\zeta}
|g(X_t,0)|\,dA^{\mu}_t\right)^q.
\end{align}
By Lemma \ref{lem2.1},
\[
E_x\int^{\zeta}_{T\wedge\zeta} |g(X_{t},0)|\,dA^{\mu}_{t} =
%E_x\int^{\infty}_T |g(X_{t},0)|\,dA^{\mu}_{t}.
E_x\int^{\infty}_T |g(X_{t},0)|\,dA^{0,\tilde\mu}_{t},
\]
%By the  Markov property and (\ref{eq3.31}),
%\begin{align*}
%E_x\int^{\infty}_T |g(\BBX_{t},0)|\,dA^{\mu}_{t}
%&=E_xE_{\BBX_T}\int^{\infty}_0 |g(\BBX_{t},0)|\,dA^{\mu}_{t}\\
%&=E_xE_{T,X_T}\int^{\infty}_0
%|g(X_{t}\circ\theta_T,0)|\,dA^{\mu}_{t},
%\end{align*}
%whereas by Lemma \ref{lem2.1},
%\begin{align*}
%E_{T,X_T}\int^{\infty}_0 |g(X_{t}\circ\theta_T,0)|\,dA^{\mu}_{t}
%&=E_{T,X_T}\int^{\infty}_0
%|g(X_{t}\circ\theta_T,0)|\,dA^{0,\tilde\mu}_t\circ\theta_T\\
%&=E_{0,X_T}\int^{\infty}_0 |g(X_{t},0)|\,dA^{0,\tilde\mu}_t.
%\end{align*}
so by the  Markov property of $\BM^{(0)}$,
%From the above  it follows that
\begin{equation}
\label{eq5.43} E_x\int^{\zeta}_{T\wedge\zeta}
|g(X_{t},0)|\,dA^{\mu}_{t} =
P_T(R_0(|g(\cdot,0)|\cdot\tilde\mu))(x).
\end{equation}
Similarly, since $\int^t_0|f(X_{s},0)|\,ds =A^{|f(\cdot,0)|\cdot
m}_t$ for $t\ge0$, we have
\begin{equation}
\label{eq5.44} E_x\int^{\zeta}_{T\wedge\zeta} |f(X_{t},0)|\,dt =
P_T(R_0|f(\cdot,0)|)(x).
\end{equation}
Combining  (\ref{eq5.42})--(\ref{eq5.44}) yields (\ref{eq5.5}) but
with constant 3 replaced by $(1-q)^{-1/q}$ with arbitrary $q\in
(0,1)$. This proves (\ref{eq5.5}) since
$(1-q)^{-1/q}\rightarrow e$ as $q\downarrow0$.
\end{proof}
\medskip

Let $\lambda\ge0$ and let $L^{\lambda}$ denote the operator
associated with the form $(B_{\lambda},V)$, i.e.
\begin{equation}
\label{eq5.27} L^{\lambda}=L^0-\lambda,
\end{equation}
where $L^0$ is the operator associated with  $(B_0,V)=(B,V)$. Let
$(P^{\lambda}_t)$, $(R^{\lambda}_{\alpha})$ denote the semigroup
and the resolvent associated with the Hunt process corresponding
to $(B_{\lambda},V)$. It is well known that for $\psi\in
L^1(E;m)$, $\mu\in\RR(E)$ we have
\[
P^{\lambda}_t\psi(x)=e^{-\lambda t}P^0_t\psi(x),\quad
R^{\lambda}_{\alpha}\mu(x)=R^0_{\alpha+\lambda}\mu(x)
\]
for q.e. $x\in E$. Therefore from Theorem \ref{th5.2} we
immediately get the following corollary.

\begin{corollary}
\label{cor5.2} Let the assumptions of Theorem \ref{th5.2} hold.
Let $u,v$ be solutions of \mbox{\rm(\ref{eq5.37})} and
\mbox{\rm(\ref{eq5.38})}, respectively, with $L=L^{\lambda}$
defined by \mbox{\rm(\ref{eq5.27})}. Then for q.e. $x\in E$,
\[
|u(T,x)-v(x)|\le 3e^{-\lambda T}\big(
P^0_T|\varphi|(x)+P^0_T(R^0_{\lambda}(|f(\cdot,0)|
+|g(\cdot,0)|\cdot\tilde\mu))(x)\big)
\]
for all $T>0$.
%Also note that if $\lambda>0$ then assumption .. is automatically
%satisfied since $P^0_t|\varphi|
\end{corollary}

\begin{remark}
The results of Sections \ref{sec3}--\ref{sec5} can be carried over
to quasi-regular forms. Indeed, if the forms $\{B(t),t\in[0,T]\}$
are quasi-regular, then by \cite[Theorem IV.2.2]{S}, there exists a
special standard process $\BBM$ properly associated in the
resolvent sense with the time dependent form defined by
(\ref{eq2.23}). One can check that all the results of Sections
\ref{sec3} and \ref{sec4} hold true for such a process. This is
because  in their proofs the fact that $\BBM$ is a Hunt process is
not used and the results of \cite{K:JFA} on which we rely in the
proofs of Section \ref{sec3} hold for quasi-regular forms
$(B^{(t)},V)$ (see \cite[Remark 4.4]{K:JFA}). Similarly, the
results of \cite{KR:CM} on which we rely in Section \ref{sec4}
hold for quasi-regular form $(B,V)$. As a consequence, Theorem
\ref{th5.2} holds true in the case of quasi regular form $(B,V)$
(its proof for such forms requires no changes).
\end{remark}

\section{Applications}
\label{sec6}

In this section, we give four quite different examples of forms
$(B,V)$ and measures $\mu$ for which Theorem \ref{th5.2} applies.

\subsection{Classical local Dirichlet forms}
\label{sec6.1}

In this subsection, we assume that $E=D$, where $D$ is a nonempty
connected bounded open subset of $\BR^d$ with $d\ge2$. We denote by $m$  the
Lebesgue measure on $D$. We consider the classical form $(B,V)$ on $H=L^2(D;m)$
defined as
\begin{equation}
\label{eq6.16}
B(\varphi,\psi)=\frac12\int_D(\nabla\varphi,\nabla\psi)\,dx,\quad
\varphi,\psi\in V.
\end{equation}
We will consider two cases: $V=H^1_0(D)$ and $V=H^1(D)$.

\paragraph{Equations with Dirichlet boundary conditions}

Let $V=H^1_0(D)$. It is well known that $(B,V)$ is a regular
Dirichlet form on $H$ (see \cite[Example 1.2.3]{FOT}). The operator
$L$ associated with $(B,V)$ in the sense of (\ref{eq2.14}) is
$\frac12\Delta$ with the Dirichlet boundary condition (see
\cite[Example 1.3.1]{FOT}). The process $\BM^{(0)}=(X,P_x)$
associated with $(B,V)$ in the resolvent sense is the  Brownian
motion killed upon leaving $D$ (see \cite[Example 4.4.1]{FOT}).
Its life time  is equal to $\tau_D=\inf\{t>0:X_t\notin D\}$.

We consider the problems
\begin{equation}
\label{eq6.1}
\partial_tu-\frac12\Delta u+h(u)|\nabla u|^2=\mu,\qquad
u|_{(0,\infty)\times\partial D}=0,\qquad u(0,\cdot)=\varphi
\end{equation}
and
\begin{equation}
\label{eq6.02} -\frac12\Delta v+h(v)|\nabla v|^2=\tilde\mu,\qquad
u|_{\partial D}=0,
\end{equation}
where $\varphi\in L^1(D;m)$ is nonnegative, $\mu=dt\otimes\tilde\mu$ with
$\tilde\mu\in\MM^+_{0,b}(D)$ and
$h:\BR\rightarrow\BR$ is a continuous function satisfying the
``sign condition", i.e.
\begin{equation}
\label{eq6.4} \forall s\in\BR,\quad h(s)s\ge0.
\end{equation}
The model example is $h(s)=s$, $s\in\BR$. In equations (\ref{eq6.1}) and (\ref{eq6.02})
gradient of the solution appears, so
they are more general than the equations studied in Sections
\ref{sec3}--\ref{sec5}. We shall see, however, that they are
closely related to equations of the forms (\ref{eq5.37}),
(\ref{eq5.38}).

We first give definitions of probabilistic solutions of
(\ref{eq6.1}), (\ref{eq6.02}).

\begin{definition}
(a) We say that $u:(0,\infty)\times D\rightarrow\BR$ is a
probabilistic solution of (\ref{eq6.1}) if for every $T>0$ the
function $\bar u$ defined as $\bar u(t,x)=u(T-t,x)$, $(t,x)\in
D_T$, is a probabilistic solution of the problem
\begin{equation}
\label{eq6.3}
\partial_t\bar u+\frac12\Delta\bar u-h(\bar u)|\nabla\bar u|^2=-\mu,\qquad
\bar u|_{(0,T)\times\partial D}=0,\qquad\bar u(T,\cdot)=\varphi,
\end{equation}
i.e. $h(\bar u)|\nabla\bar u|^2\in\RR(D_{0,T})$ and for q.e. $z\in
D_{0,T}$,
\begin{equation}
\label{eq6.15} \bar u(z)=E_z\left(\varphi(\BBX_{\zeta_{\tau}})
-\int^{\zeta_{\tau}}_0 h(\bar u)(\BBX_t)|\nabla \bar
u(\BBX_t)|^2\,dt+\int^{\zeta_{\tau}}_0dA^{\mu}_t\right).
\end{equation}
(b) We say that $v:D\rightarrow\BR$ is a probabilistic solution of
(\ref{eq6.02}) if $h(v)|\nabla v|^2\in\RR(E)$ and for q.e. $x\in
D$,
\[
v(x)=E_x\left(-\int^{\zeta}_0 h(v)(X_t)|\nabla
v(X_t)|^2\,dt+\int^{\zeta}_0dA^{0,\tilde\mu}_t\right).
\]
\end{definition}

Let
\[
G(s)=2\int^s_0h(t)\,dt,\qquad \Phi(s)=\int^s_0\exp(-G(t))\,dt,
\quad s\in\BR,
\]
and let $H:\Phi(\BR)\rightarrow\BR$ be defined as
\[
H(s)=\exp(-G(\Phi^{-1}(s))).
\]
The function $\Phi$ is strictly increasing on $\BR$, and by (\ref{eq6.4}), $G$ is nondecreasing on $[0,\infty)$.  We set $\Phi(\infty)=\lim_{s\rightarrow\infty}\Phi(s)$, $G(\infty)=\lim_{s\rightarrow\infty}G(s)$, and we define $\hat H:\BR\rightarrow\BR$ by
\[
\begin{cases}
\hat H(s)=H(s), s\in[0,\infty),& \mbox{if } \Phi(\infty)=\infty,\\
\hat H(s)=H(s), s\in[0,\Phi(\infty)]\mbox{ and }\hat H(s)=e^{-G(\infty)}, s>\Phi(\infty),&\mbox{if } \Phi(\infty)<\infty,\\
\hat H(s)=H(0), &\mbox{if }s<0.
\end{cases}
\]
Notice that $\hat H$ is continuous and  nonincreasing on $\BR$, and $0\le\hat H\le 1$. Therefore $g:=\hat H$ satisfies the hypotheses (E2), (E5) and (E6).

In Proposition \ref{prop6.2} below we show that probabilistic
solutions of problems (\ref{eq6.1}), (\ref{eq6.02}) are closely
related to the probabilistic solutions of problems
\begin{equation}
\label{eq6.2}
\partial_tw-\frac12\Delta w=\hat H(w)\cdot\mu, \qquad w|_{(0,\infty)\times\partial
D}=0,\qquad w(0,\cdot)=\Phi(\varphi)
\end{equation}
and
\begin{equation}
\label{eq6.04} -\frac12\Delta\tilde w=\hat H(\tilde w)\cdot\tilde\mu,\qquad
\tilde w|_{\partial D}=0.
\end{equation}
We start with the observation that in fact,  in the above equations, one can replace $\hat H$ by $H$.

\begin{remark}
\label{rem6.1}
If $w$ is a solution of (\ref{eq6.2}), then $0\le w\le \Phi(\infty)$ q.e. on $(0,\infty)\times D$. Thus, we can replace $\hat H$ by $H$ in (\ref{eq6.2}). Similarly, if  $\tilde w$ is a solution of (\ref{eq6.04}), then $0\le\tilde w\le \Phi(\infty)$ q.e. on $D$. Thus, we can replace $\hat H$ by $H$ in (\ref{eq6.04}).

We provide the proof for (\ref{eq6.2}). The proof for (\ref{eq6.04}) is similar. Let  $T>0$ and $\bar w(t,x)=w(T-t,x)$. By \cite[Proposition 3.7]{KR:JEE}, for q.e. $z\in D_{0,T}$ the pair
\[
(Y_t,Z_t)=(\bar w(\BBX_t),\nabla\bar w(\BBX_t)),\quad
t\in[0,\zeta_{\tau}],
\]
is  a solution of the BSDE
\begin{equation}
\label{eq6.7} Y_t=\Phi( \varphi(\BBX_{\zeta_{\tau}}))
+\int_{t\wedge\zeta_{\tau}}^{\zeta_{\tau}}\hat H(Y_{s})\,dA^{\mu}_{s}
-\int^{\zeta_{\tau}}_{t\wedge\zeta_{\tau}}Z_{s}\,dW_{s},\quad
t\in[0,\zeta_{\tau}],
\end{equation}
under the measure $P_z$, where $W$ is some Wiener process starting
from $z$ under $P_z$ (In different words, in the case where the form
(\ref{eq6.16}) is considered, if $w$ is a probabilistic solution of
(\ref{eq6.2}), then the martingale $M$
appearing in Theorem \ref{th3.1} (with
the data from (\ref{eq6.2})) has the representation
$M_t=\int^t_0Z_r\,dW_r$ with $Z$ as above). Since, by assumption, $\varphi\ge0$, we have $\Phi\circ\varphi
\ge0$, so from (\ref{eq6.7}) it follows that $\bar w\ge0$ q.e. on $D_{0,T}$. Since $T>0$ was arbitrary, $w\ge0$ q.e. on $(0,\infty)\times D$. Since $w$ is quasi-continuous, it is finite q.e., so $w\le\Phi(\infty)$ q.e.  on $(0,\infty)\times D$ if $\Phi(\infty)=\infty$. Suppose now that $\Phi(\infty)<\infty$. To show that $w\le\Phi(\infty)$, we first assume additionally that
\begin{equation}
\label{eq6.10}
\int_0^\infty h(s)\,ds=\infty.
\end{equation}
Choose $\{a_n\}\subset[0,\infty)$ such that  $a_n\nearrow\Phi(\infty)$, By (\ref{eq6.7}) and the Meyer-Tanaka formula, for q.e. $z\in E_{0,T}$ we have
\begin{align}
\label{eq6.21}
(\bar w(z)-a_n)^+&\le E_z\left(\Phi\circ\varphi(X_T)-a_n)^{+}
+\int_0^{\zeta_{\tau}}\fch_{\{\bar w(\BBX_s)>a_n\}}\hat H(\bar w(\BBX_s))\,dA^\mu_s\right)
\nonumber\\
&\le E_z\left(\Phi\circ\varphi(X_T)-a_n)^{+}
+\int_0^{\zeta_{\tau}}H(a_n)\,dA^\mu_s\right).
\end{align}
By (\ref{eq6.10}), $H(a_n)\searrow 0$, so letting $n\rightarrow\infty$ in (\ref{eq6.21}) yields $(\bar w(z)-\Phi(\infty))^+=0$. Since $T>0$ was arbitrary, this implies that  $w\le\Phi(\infty)$ q.e. on $(0,\infty)\times D$.
We now show how to dispense with the assumption (\ref{eq6.10}). Let $h_n(x)=h(x)+(1/n) \arctan x$ and  $w_n$ be a solution  of  (\ref{eq6.2}) with $\hat H$ replaced by $\hat H_n$ defined as $\hat H$ but with $h$ replaced by $h_n$. By what has already been proved $w_n\le\Phi(\infty)$ q.e. on $(0,\infty)\times D$. Set $\bar w_n(t,x)=w_n(T-t,x)$. By using estimates of the form (\ref{eq3.33}) we show that $\bar w_n\rightarrow\bar w$ q.e. on $D_{0,T}$. Consequently, $\bar w\le \Phi(\infty)$ q.e. on $D_{0,T}$, so $w\le\Phi(\infty)$ q.e. on $(0,\infty)\times D$.
\end{remark}

Assertions (ii) and (iii) of Proposition \ref{prop6.2}
may be viewed as probabilistic reformulation of known analytic
facts relating (\ref{eq6.1}), (\ref{eq6.02}) to (\ref{eq6.2}),
(\ref{eq6.04}) (see, e.g., \cite{LP}) or \cite[Remark 2.17]{MP}).

\begin{proposition}
\label{prop6.2} Assume that $\varphi\in L^1(D;m)$ is nonnegative,
$\mu\in\MM^+_{0,b}(D)$ and $h:D\rightarrow \BR$ is a continuous
function satisfying \mbox{\rm(\ref{eq6.4})}. Then
\begin{enumerate}
\item[\rm(i)]There exists a unique solution $u$  of problem
\mbox{\rm(\ref{eq6.1})} and a unique solution $v$ of problem \mbox{\rm(\ref{eq6.02})}. Moreover, $0\le u\le\Phi(\infty)\ge0$ q.e. on $(0,\infty)\times D$ and $0\le v\le\Phi(\infty)$ q.e. on $D$.
\item[\rm(ii)]$u$ is a probabilistic solution of \mbox{\rm(\ref{eq6.1})}
if and only if $w=\Phi(u)$ is a solution of
\mbox{\rm(\ref{eq6.2})}.
\item[\rm(iii)]$v$ is a probabilistic solution of \mbox{\rm(\ref{eq6.02})}
if and only if $\tilde w=\Phi(v)$ is a solution of
\mbox{\rm(\ref{eq6.04})}.
\end{enumerate}
\end{proposition}
\begin{proof}
We first prove (ii). Let $w$ be a solution of (\ref{eq6.2}). For fixed  $T>0$, we define $\bar w$ and $(Y,Z)$ as in Remark \ref{rem6.1}. We know that $0\le\bar w\le\Phi(\infty)$ q.e. on $D_{0,T}$. Let $\bar u=\Phi^{-1}(\bar w)$. Since $\Phi^{-1}$ is of class $C^2$,
applying It\^o's formula we get
\begin{align*}
\bar u(\BBX_{\zeta_{\tau}}) -\bar u(\BBX_{0})
&=\Phi^{-1}(Y_{\zeta_{\tau}})-\Phi^{-1}(Y_{0})\\
&=\int^{\zeta_{\tau}}_{0} (\Phi^{-1})'(Y_{t})\,dY_{t}
+\frac12\int^{\zeta_{\tau}}_{0} (\Phi^{-1})''(Y_{t})\,d\langle
Y\rangle_{t}.
\end{align*}
But
\[
(\Phi^{-1})'=\frac1{\Phi'(\Phi^{-1})}\,,\qquad
(\Phi^{-1})''=-\frac1{(\Phi'(\Phi^{-1}))^2}\cdot
\Phi''(\Phi^{-1})\cdot\frac1{\Phi'(\Phi^{-1})}
\]
a.e. with respect to the Lebesgue measure. Hence
\begin{align*} \bar
u(\BBX_0)&=\bar u(\BBX_{\zeta_{\tau}})
+\int^{\zeta_{\tau}}_{0}\frac1{\Phi'(\bar u)}\,
H(\bar w(\BBX_{t}))\,dA^{\mu}_{t}\\
&\quad-\int^{\zeta_{\tau}}_{0}\frac1{\Phi'(\bar u)}\,\nabla \bar
w(\BBX_{t})\,dW_{t} +\frac12\int^{\zeta_{\tau}}_{0}
\frac{\Phi''(\bar u)}{(\Phi'(\bar u))^3}\,|\nabla\bar w|^2
(\BBX_{t})\,dt.
\end{align*}
Since $\bar u(\BBX_{\zeta_{\tau}})= \varphi(\BBX_{\zeta_{\tau}})$
and
\[
\frac{\Phi''}{(\Phi')^3}=-\frac{2h}{(\Phi')^2}\,,\qquad\Phi'(\bar
u)=H(\bar w),\qquad \nabla\bar  w=\Phi'(\bar u)\nabla\bar u,
\]
we have
\begin{align*}
\bar u(\BBX_0)&=\varphi(\BBX_{\zeta_{\tau}})
+\int^{\zeta_{\tau}}_{0}dA^{\mu}_{t}
-\int^{\zeta_{\tau}}_{0}h(\bar u)|\nabla\bar u|^2(\BBX_{t})\,dt
-\int^{\zeta_{\tau}}_{0}\nabla \bar u(\BBX_t)\,dW_{t}.
\end{align*}
Taking the expectation with respect to $P_x$ we see that $\bar
u=\Phi^{-1}(\bar w)$ is a probabilistic solution of (\ref{eq6.3}).
Hence $u=\Phi^{-1}(w)$ is a probabilistic solution of
(\ref{eq6.1}).

To prove the opposite implication, we first note that if $u$
is a solution of (\ref{eq6.1}), then for every $T>0$, for q.e. $z\in D_{0,T}$ the
pair
\[
(\tilde Y_t,\tilde Z_t)=(\bar u(\BBX_t),\nabla\bar
u(\BBX_t)),\quad t\in[0,\zeta_{\tau}],
\]
is  a solution of the BSDE
\[
\tilde Y_t=\varphi(\BBX_{\zeta_{\tau}})
-\int_{t\wedge\zeta_{\tau}}^{\zeta_{\tau}}h (\tilde Y_{s})|\tilde
Z_{s}|^2\,ds +\int_{t\wedge\zeta_{\tau}}^{\zeta_{\tau}}dA^{\mu}_{s}
-\int^{\zeta_{\tau}}_{t\wedge\zeta_{\tau}}Z_{s}\,dW_{s},\quad
t\in[0,\zeta_{\tau}],
\]
under the measure $P_z$. In case $h(\bar u)|\nabla\bar u|^2\in
L^1(D_{0,T};m_1)$ this follows directly from \cite[Proposition
3.7]{KR:JEE}, while in case $h(\bar u)|\nabla\bar u|^2\cdot
m\in\RR(D_{0,T})$ follows from \cite[Proposition 3.7]{KR:JEE} by
simple approximation. Put $\bar w =\Phi(\bar u)$. Applying It\^o's
formula we show that the pair
\[
(Y_t,Z_t)=(\bar w(\BBX_t),\nabla\bar w(\BBX_t)),\quad
t\in[0,\zeta_{\tau}],
\]
is a solution of (\ref{eq6.7}). From this it follows that $w$ is a solution of (\ref{eq6.2}).
This completes the proof of (ii).

The proof of (iii) is similar to that of (ii). We apply It\^o's
formula and the fact that in case of the form (\ref{eq6.16}), the
martingale $M$ appearing in  Theorem \ref{th4.3} has the representation
$M_t=\int^t_0Z_s\,dW_s$, $t\ge0$, with $Z_t=\nabla v(X_t)$ if we
consider equation (\ref{eq6.02}), and with $Z_t=\nabla \tilde
w(X_t)$ if we consider (\ref{eq6.04}) (for the representation
property for $M$ see \cite[Theorem 3.5]{K:AMPA}.

We now show (i). We know that  $g:=\hat H$ satisfies the hypotheses (E2), (E5) and (E6). Therefore, by Theorem \ref{th3.1}, there exists a unique solution $w$
of (\ref{eq6.2}), while by Theorem \ref{th4.3}, there exists a unique solution $\tilde w$ of (\ref{eq6.04}). Therefore (i) follows from (ii), (iii) and Remark \ref{rem6.1}.
\end{proof}

\begin{remark}
Assume that  $\varphi\in L^1(D)$ is nonnegative, $\tilde\mu(dx)=\beta(x)\,dx$ for some
$\beta\in L^1(D)$ and  $h$ is a continuous function satisfying
(\ref{eq6.4}). Moreover, assume that there exist $L,\delta>0$ such
that $h(s)s\ge\delta$ for $s\in\BR$ such that $|s|\ge L$.
\smallskip\\
(i) In \cite{BG} it is proved that under the above assumptions
there exists a weak solution $v\in H^1_0(D)$ of (\ref{eq6.02})
such that $h(v)|\nabla v|^2\in L^1(D;m)$. A quasi-continuous
version of $v$, which we still denote by $v$, is a probabilistic
solution of (\ref{eq6.02}). Indeed, since for every bounded $w\in
H^1_0$ we have $B(v,w)=\int_D(h(v)|\nabla v|^2+\beta)w\,dx$, $v$
is a solution of problem (\ref{eq6.02}) in the sense of duality
(see \cite[Section 5]{KR:JFA} for the definition). Therefore, by
\cite[Proposition 5.1]{KR:JFA}, $v$ is a probabilistic solution of
(\ref{eq6.02}).
\smallskip\\
(ii) By the results proved in \cite{Po}, there exists a  weak solution
$\bar u\in L^2(0,T;H^1_0(D))$ of problem (\ref{eq6.3}) such that
$h(\bar u) |\nabla\bar u|^2\in L^1(D_T;m_1)$. Its quasi-continuous
version is a probabilistic solution of (\ref{eq6.3}). This follows
from the fact that it is a solution of (\ref{eq6.3}) in the sense
of duality (see \cite[Section 4]{K:JFA} for the definition), and
hence, by \cite[Corollary 4.2]{K:JFA}, a probabilistic solution of
(\ref{eq6.3}).
\end{remark}

\begin{proposition}
Let $\varphi,h$ satisfy  the assumptions of Proposition \ref{prop6.2}, and let $\mu(dx)=\beta(x)\,m(dx)$ for some nonnegative $\beta\in
L^1(D;m)$. Then
\begin{enumerate}
\item[\rm(i)]For q.e. $x\in D$, $u(t,x)\rightarrow v(x)$ as $t\rightarrow\infty$.
\item[\rm(ii)]$u(t,\cdot)\rightarrow v$ in $L^1(D;m)$ as $t\rightarrow\infty$.
\end{enumerate}
\end{proposition}
\begin{proof}
In the proof we adopt the notation from the proof of  Proposition \ref{prop6.2}.
We know that $w=\Phi(u)$ is nonnegative and solves  (\ref{eq6.1}) with $H$ replaced by $\hat H$. We also know that the initial condition $\Phi\circ\varphi$ and coefficients $f=0, g:=\hat H$ of that equation satisfy the assumptions (E1)--(E6). Moreover, we shall see in the
proof of Proposition \ref{prop6.1} (in a more general situation
where $\Delta$ is replaced by the fractional Laplacian
$\Delta^{\alpha/2}$) that (\ref{eq5.45}) with $\varphi$ replaced by $\Phi\circ\varphi$ is satisfied. Hence, by Theorem \ref{th5.2}, $w(t,x)\rightarrow\tilde w(x)$ as
$t\rightarrow\infty$ for q.e. $x\in D$. Therefore part (i) follows
from Proposition \ref{prop6.2} and the fact that $\Phi^{-1}$ is
continuous. To prove part (ii), we first note that for every $T>0$, $\bar w\ge0$ q.e. on $D_{0,T}$, so $\bar u\ge0$ q.e. on $D_{0,T}$. Consequently, $h(\bar u)\ge0$ q.e. on $D_{0,T}$ since $h$ satisfies (\ref{eq6.4}). Therefore from (\ref{eq6.15}) it follows that for
q.e. $(s,x)\in D_{0,T}$,
\[
\bar u(s,x)\le E_{s,x}\left(\varphi(\BBX_{\zeta_{\tau}})
+\int^{\zeta_{\tau}}_0dA^{\mu}_t\right)=:\bar{\hat u}(s,x).
\]
The function ${\hat u}$ defined as $\hat u(t,x)=\bar{\hat
u}(T-t,x)$, $(t,x)\in D_T$, is a solution of (\ref{eq6.1}) with
$h\equiv0$. By Theorem \ref{th5.2}, $\hat u(t,x)\rightarrow\hat
v(x)$ as $t\rightarrow\infty$ for q.e. $x\in D$, where $\hat v$ is
a solution of (\ref{eq6.02}) with $h\equiv0$. In fact, by
(\ref{eq5.5}) (see the proof of Proposition \ref{prop6.1} for
details),
\[
|\hat u(t,x)-\hat v(x)|\le Ct^{-d/2}(\|\varphi\|_{L^1(D;m)}
+\|\beta\|_{L^1(D;m)}),\quad t>0,
\]
for q.e. $x\in D$. Since $D$ is bounded, it follows that
$\hat u(t,\cdot)\rightarrow\hat v$ in
$L^1(D;m)$ as $t\rightarrow\infty$. From this and the fact that
$0\le u(t,\cdot)\le\hat u(t,\cdot)$ we conclude that the family
$\{u(t,\cdot)\}$ is uniformly integrable, which together with (i)
proves  (ii).
\end{proof}
\medskip

By using a completely different method, part (ii) of the above
proposition was proved in \cite[Theorem 3.3]{LP} under the
assumption that $h\in C^1(\BR)$ and $h'(s)>0$ for $s\in\BR$.

\paragraph{Equations with Neumann boundary conditions}

Let $D$ be a bounded Lipschitz domain in $\BR^d$, $d\ge3$. Set
$E=\bar D$. Let $H=L^2(\bar D;m)$, where $m$ is the Lebesgue
measure on $\bar D$, and let $V=H^1(D)$. It is known that $(B,V)$
defined by (\ref{eq6.16}) is a regular Dirichlet form on $H$ (see
\cite[Example 4.5.3]{FOT}). The operator $L$ associated with
$(B,V)$ in the sense of (\ref{eq2.14}) is $\frac12\Delta$ with the
Neumann boundary condition, while the process $\BM^{(0)}=(X,P_x)$
(with life time $\zeta=\infty$) associated with $(B,V)$ in the
resolvent sense is the reflecting Brownian motion on $\bar D$ (see
\cite[Example 4.5.3]{FOT}).

Let $\tilde\nu$ denote the surface measure on $\partial D$. Then  for
$\tilde\nu$-a.e. $x\in\partial D$ there exists a unit inward normal
vector $\mathbf{n}(x)=(\mathbf{n}_1(x),\dots,\mathbf{n}_d(x))$
(see \cite[Example 5.2.2]{FOT}). We consider the Neumann problems
\begin{equation}
\label{eq6.01}
\partial_su-\frac12\Delta u+\lambda u=f(\cdot,u), \qquad
\frac{\partial u}{\partial\mathbf{n}}\Bigl|_{(0,\infty)\times\partial
D} =g(x,u),\qquad u(0,\cdot)=\varphi
\end{equation}
and
\begin{equation}
\label{eq6.17} -\frac12\Delta v+\lambda v=f(\cdot,v),\qquad
\frac{\partial v}{\partial\mathbf{n}}\Bigl|_{\partial D}=g(\cdot,v),
\end{equation}
where $\frac{\partial u}{\partial\mathbf{n}}
=\sum^d_{i=1}\mathbf{n}_i\frac{\partial u}{\partial x_i}$\,. It is
known (see \cite[Example 5.2.2]{FOT}) that for every $x\in\bar D$
the process $X$ has under $P_x$ the representation
\begin{equation}
\label{eq6.18}
X^{i}_t=X^{i}_0+B^i_t+\frac12\int^t_0\mathbf{n}_i(X_s)\,dl_s,\quad
t\ge0, \quad P_x\mbox{-a.s.},
\end{equation}
where $B=(B^1,\dots,B^d)$ is a  $d$-dimensional standard Brownian
motion and $l$ is the local time of $X$ on the boundary $\partial
D$. Let $S_{00}(\bar D)$ denote the set of all finite positive Borel
measures $\gamma$ on $\bar D$ of finite energy integrals  and such that
$\|U_1\gamma\|<\infty$, where $U_1\gamma$ is the 1-potential of $\gamma$.
It is also known (see \cite[Example 5.2.2]{FOT}) that $\tilde\nu\in S_{00}(\bar D)$
and
\begin{equation}
\label{eq6.19} A^{0,\tilde\nu}=l.
\end{equation}
By (\ref{eq6.18}) and (\ref{eq6.19}), the probabilistic solution of
(\ref{eq6.01}) (see \cite[Section 4]{PZ}) coincides with the
probabilistic solution of
\begin{equation}
\label{eq6.5}
\partial_t u-Lu+\lambda u=f(x,u)
+g(x,u)\cdot\nu,\quad u(0,\cdot)=\varphi
\end{equation}
with $\nu=dt\otimes\tilde\nu$, and the probabilistic solution of  (\ref{eq6.17}) (see
\cite[Section 5]{PZ}) coincides with the probabilistic solution of
\begin{equation}
\label{eq6.6} -Lw+\lambda w=f(x,w) +g(x,w)\cdot\tilde\nu.
\end{equation}

\begin{proposition}
\label{prop6.4} Let $\varphi,f,g$ satisfy the assumptions of
Theorem \ref{th4.3}, and moreover, $f(\cdot,0)\in L^1(D;m)$,
$g(\cdot,0)\in L^{\infty}(D;m)$. Let $u$ be a solution of
\mbox{\rm(\ref{eq6.5})} and $v$ be a solution of
\mbox{\rm(\ref{eq6.6})}. Then for every $\lambda>0$ there is $C>0$
depending only on $d$  such that for q.e. $x\in\bar D$,
\begin{align*}
|u(t,x)-v(x)|&\le Ce^{-\lambda t}t^{-d/2}
\Big(\|\varphi\|_{L^1(D;m)}+\lambda^{-1}\|f(\cdot,0)\|_{L^1(D;m)}\nonumber\\
&\qquad\qquad\qquad\quad
+(1\vee\lambda^{-1})m(D)\|g(\cdot,0)\|_{\infty}
\|R^0_1\tilde\nu\|_{\infty}\Big),\quad t>0.
\end{align*}
\end{proposition}
\begin{proof}
By \cite[Theorem 3.1]{BH} (see also \cite[Lemma 4.3]{BH}), there is
$C>0$ depending only on $d$ such that for every $\psi\in
L^1(D;m)$,
\[
\sup_{x\in\bar D}P^0_t|\psi|(x)\le
Ct^{-d/2}\|\psi\|_{L^1(D;m)},\quad t>0.
\]
Moreover,
\[
\|R^0_{\lambda}f(\cdot,0)\|_{L^1(D;m)}=(f(\cdot,0),R^0_{\lambda}1)
=\lambda^{-1}\|f(\cdot,0)\|_{L^1(D;m)}.
\]
Since $\tilde\nu\in S_{00}(\bar D)$,  $\|R^0_1\tilde\nu\|_{\infty}<\infty$. By the
resolvent equation (see \cite[Lemma 5.1.5]{FOT}),
\[
R^0_{\lambda}\tilde\nu=R^0_1\tilde\nu +(1-\lambda)R^0_{\lambda}(R^0_1\tilde\nu).
\]
Hence $\|R^0_{\lambda}(g(\cdot,0)\cdot\tilde\nu)\|_{L^1(D;m)}
\le\|g(\cdot,0)\|_{\infty}\|R^0_1\tilde\nu\|_{L^1(D;m)} \le m(D)
\|g(\cdot,0)\|_{\infty}\|R^0_1\tilde\nu\|_{\infty}$ if $\lambda\ge1$ and
\begin{align*}
\|R^0_{\lambda}(g(\cdot,0)\tilde\nu)\|_{L^1(D;m)}&\le
m(D)\|g(\cdot,0)\|_{\infty}(\|R^0_1\tilde\nu\|_{\infty}
+(1-\lambda)\lambda^{-1}\|R^0_1\tilde\nu\|_{\infty})\\
&=\lambda^{-1}m(D)\|g(\cdot,0)\|_{\infty}\|R^0_1\tilde\nu\|_{\infty}
\end{align*}
if $\lambda<1$. The proposition follows immediately from the above
estimates and Corollary \ref{cor5.2}.
\end{proof}

\subsection{Nonlocal Dirichlet forms}
\label{sec6.2}

Let $E=\BR^d$ with $d\ge2$, $m$ be the Lebesgue measure on $E$ and
$\alpha\in(0,2)$. We consider the form
\begin{equation}
\label{eq6.23}
B(u,v)=\int_{\BR^d}\hat{u}(x)\hat{v}(x)|x|^{\alpha}\,dx,\quad
u,v\in V,
\end{equation}
where $\hat u$ denotes the Fourier transform of $u$ and
\[
V=\left\{u\in L^{2}(\BR^d;m):
\int_{\BR^d}|\hat{u}(x)|^{2}\,|x|^{\alpha}\,dx<\infty\right\}.
\]
It is known  that $(B,V)$ is a regular Dirichlet form on
$L^2(\BR^d;m)$ (see \cite[Example 1.4.1]{FOT}).  The operator $L$
associated with $(B,V)$ is the fractional Laplacian
$\Delta^{\alpha/2}$ and the Markov process $\BM^{(0)}=(X,P_x)$
(with life time $\zeta=\infty$) associated with $(B,V)$ is a
symmetric stable process of index $\alpha$.

Let $D\subset\BR^d$, $d\ge2$, be a nonempty open bounded connected
set. Set $L^2_D(\BR^d;m)=\{u\in L^2(\BR^d;m):u=0\mbox{ a.e. on
}D^c\}$, $V_D=\{u\in D(B):\tilde u=0\mbox{ q.e. on }D^c\}$, where
$\tilde u$ is a quasi-continuous version of $u$. By \cite[Theorem
4.4.3]{FOT}, the form $(B,V_D)$ is a regular Dirichlet form on
$L^2_D(\BR^d;m)$, and by \cite[Theorem 4.4.4]{FOT}, if $(B,V)$ is
transient, then $(B,V_D)$ is transient, too.

%In what follows we denote by $\BM^{(0)}_D$ the part of the process
%$\BM^{(0)}$ on $D$ (see \cite[Section 4.4]{FOT}), by $\zeta_D$ its
%life time and by $(P^D_t)$ the semigroup associated with
%$\BM^{(0)}_D$.

\begin{proposition}
\label{prop6.1} Let $\varphi,f,g$ satisfy the assumptions of
Theorem \ref{th4.3}, and moreover, $f(\cdot,0)\in L^1(D;m)$,
$g(\cdot,0)\cdot\tilde\mu\in\MM_{0,b}(D)$. Let $u$ be a solution
of \mbox{\rm(\ref{eq5.37})} and $v$ be a solution of
\mbox{\rm(\ref{eq5.38})}. Then there exists $C>0$ depending only
on $d,\alpha$ such that q.e. $x\in D$,
\begin{align}
\label{eq6.13} |u(t,x)-v(x)|&\le Ct^{-d/\alpha}
\left(\|\varphi\|_{L^1(D;m)}+(m(D))^{\alpha/d}\|f(\cdot,0)\|_{L^1(D;m)}\right.
\nonumber\\
&\qquad\qquad\quad
\left.+(m(D))^{\alpha/d}(|g(\cdot,0)|\cdot\tilde\mu)(D)\right),\quad t>0.
\end{align}
\end{proposition}
\begin{proof}
Let $\BM^{(0)}_D$ denote the part of the process $\BM^{(0)}$ on
$D$ (see \cite[Section 4.4]{FOT}), $\zeta_D$ denote the life time
of  $\BM^{(0)}_D$ and let $(P^0_t)$, $(R^0_{\alpha})$ denote the
semigroup and the resolvent associated with $\BM^{(0)}_D$. We denote by  $p$
the transition density of the process $\BM^{(0)}$.
From the fact that $p(t,x,y)=p(t,0,x-y)$ and the scaling property
$p(t,0,x)=t^{-d/\alpha}p(1,0,t^{-1/\alpha}x)$ it follows that
\begin{equation}
\label{eq6.8} p(t,x,y)\le Ct^{-d/\alpha},\quad t>0
\end{equation}
with $C=\sup_{x\in\BR^d}p(1,0,x)$. Hence
\begin{equation}
\label{eq6.9} P^0_t\varphi(x)\le
Ct^{-d/\alpha}\|\varphi\|_{L^1(D;m)},\quad t>0.
\end{equation}
By (\ref{eq6.8}) and \cite[Theorem 1]{C} (see also the proof of
\cite[Theorem 1.17]{CZ}),
\[
\sup_{x\in D}E_x\zeta^0_D\le c(m(D))^{\alpha/d}
\]
for some $c>0$ depending only on $\alpha,d$.  By (\ref{eq6.9}),
\[
P^0_t(R^0_{0}(|g(\cdot,0)|\cdot\tilde\mu))(x)\le
Ct^{-d/\alpha}\|R^0_{0}(|g(\cdot,0)|\cdot\tilde\mu)\|_{L^1(D;m)}.
\]
Since
\begin{align*}
\|R^0_{0}(|g(\cdot,0)|\cdot\tilde\mu)\|_{L^1(D;m)}&
=\int_DR^0_01(x)|g(x,0)|\,\tilde\mu(dx)\\
&=\int_DE_x\zeta^0_D|g(x,0)|\,\tilde\mu(dx) \le
c(m(D))^{\alpha/d}(|g(\cdot,0)|\cdot\tilde\mu)(D),
\end{align*}
we have
\begin{equation}
\label{eq6.11} P^0_t(R^0_{0}(|g(\cdot,0)|\cdot\tilde\mu))(x)\le
c(\alpha,d)(m(D))^{\alpha/d}t^{-d/\alpha}
(|g(\cdot,0)|\cdot\tilde\mu)(D).
\end{equation}
Putting $g=1$ and $\mu=f(\cdot,0)\cdot m$ in the above estimate we
get
\begin{equation}
\label{eq6.12} P^0_t(R^0_{0}|f(\cdot,0)|)(x)\le
c(\alpha,d)(m(D))^{\alpha/d}t^{-d/\alpha}\|f(\cdot,0)\|_{L^1(D;m)}.
\end{equation}
Substituting (\ref{eq6.9})--(\ref{eq6.12}) into (\ref{eq5.5}) we
get the desired estimate.
\end{proof}

Assume additionally that $D$ has a $C^{1,1}$ boundary and
$d\ge3$. Then, by  \cite[Proposition 4.9]{Ku}, there exist constants
$0<c_1<c_2$ depending only on $d,\alpha,D$ such that
\[
c_1\delta^{\alpha/2}(x)\le R^0_01(x)\le
c_2\delta^{\alpha/2}(x),\quad x\in D,
\]
where $\delta(x)=\mbox{dist}(x,\partial D)$. It follows that if
\begin{equation}
\label{eq6.14}
\int_D\delta^{\alpha/2}(x)|g(\cdot,0)|\,\tilde\mu(dx)=: K<\infty,
\end{equation}
then (\ref{eq6.12}) holds with $|\mu(D)|$ replaced by $K$.
Therefore under the above assumptions on $D$ the proof of
Proposition \ref{prop6.1} shows the following proposition.

\begin{proposition}
Let the assumptions of Proposition \ref{prop4.1} hold, and
moreover $f\in L^1(D;m)$, $|g(\cdot,0)|\cdot\tilde\mu$ satisfies
\mbox{\rm(\ref{eq6.14})}. Then \mbox{\rm(\ref{eq6.13})} holds true
with $(m(D))^{\alpha/d}(|g(\cdot,0)|\cdot\mu)(D)$ replaced by $K$.
\end{proposition}

\begin{remark}
(i) An analogue of Proposition \ref{prop6.1} holds true for $D$ as
before and the  form (\ref{eq6.23}) replaced by any regular transient
symmetric Dirichlet form $(B,V)$ on $L^2(\BR^d;dx)$ whose
semigroup possesses a kernel $p$ satisfying uniform estimate of
the form (\ref{eq6.8}) with $\alpha/d$ replaced by $\kappa$, i.e.
\begin{equation}
\label{eq6.25} p(t,x,y)\le Ct^{-\kappa},\quad t>0,
\end{equation}
for some $C,\kappa>0$. Indeed, an inspection of the proof of
Proposition \ref{prop6.1} shows that for such a form estimate
(\ref{eq6.13}) holds with $\alpha/d$ replaced by $\kappa$. A
characterization of symmetric Dirichlet forms satisfying
(\ref{eq6.25}) in terms of Dirichlet form inequalities of Nash's
type is given in \cite{CKS}. For a concrete example of a class of
forms satisfying (\ref{eq6.25}) and containing the form
(\ref{eq6.23})  as a special case see \cite[Remark 2.15]{CKS}. For
similar examples
% involving operators corresponding to stable-like processes
see \cite[Examples 6.7.14, 6.7.16]{CF}.
\end{remark}

\subsection{Local semi-Dirichlet forms}

Let $D\subset\BR^d$, $m$, $H$ be as in Section \ref{sec6.1},  and
let $a:D\rightarrow\BR^d\otimes\BR^d$, $b:D\rightarrow\BR^d$ be
measurable functions such that for every $x\in D$,
\[
\lambda^{-1}|\xi|^2\le\sum^d_{i,j=1}a_{ij}(x)\xi_i\xi_j\le\lambda|\xi|^2,
\quad a_{ij}(x)=a_{ji}(x),\quad \sum^d_{i=1}|b_i(x)|^2\le\lambda
\]
for some $\lambda\ge1$. Set $V=H^1_0(D)$ and
\[
B(\varphi,\psi)=\sum^d_{i,j=1}\int_Da_{ij}(x)
\frac{\partial\varphi}{\partial x_i} \frac{\partial\psi}{\partial
x_i}\,dx +\sum^d_{i=1}\int_Db_i(x)\frac{\partial\varphi}{\partial
x_i}\psi(x)\,dx,\quad \varphi,\psi\in V.
\]
Of course, the operator $L$ determined by $(B,V)$ has the form
\begin{equation}
\label{eq6.22} L=\sum^d_{i,j=1}\frac{\partial}{\partial
x_i}\left(a_{ij}(x) \frac{\partial}{\partial x_j}\right)
+\sum^d_{i=1}b_i(x)\frac{\partial}{\partial x_i}\,.
\end{equation}
By \cite[Theorem 1.5.3]{O2}, $(B,V)$ is a regular lower bounded
semi-Dirichlet form on $H$. Let $G_D$ denote the Green function for
$L$ on $D$ and $\bar G_D$ denote the Green function on $D$ for
the Laplace operator $\Delta$.
%operator defined by (\ref{eq6.22}) with $b=0$.
From Aronson's estimates (see \cite{A}) it follows that there is $c>0$
such that $G_D\le c\bar G_D$. Hence, if $\mu\in\MM^+_{0,b}(D)$, then
\begin{align*}
(R_0\mu,1)=\int_D\left(\int_DG_D(x,y)\,\mu(dy)\right)dx&\le c
\int_D\left(\int_D\bar G_D(y,x)\,dx\right)\mu(dy)\\
&\le c\|\bar G_D1\|_{\infty}\mu(D),
\end{align*}
which is bounded because $\|\bar G_D1\|_{\infty}\le c'(m(D))^{2/d}$ for some $c'>0$  (see, e.g., \cite[Theorem 1]{C}). This shows that $\MM^+_{0,b}(D)\subset\RR^+(E)$.
%$(\hat R_{\alpha})$ denote the
%dual resolvent associated with $(B,V)$ and $(\bar R_{\alpha})$ denote
%$the resolvent associated with the form defined by (\ref{eq6.27})  with $b=0$.
%It is know (see ..) that there is $c>0$ such that  $\hat R_01\le c\bar R_0$.
%Hence, if $\mu\in\MM^+_{0,b}(D)$, then
%$(R_0\mu,1)=\int_D\hat R_0\mu(x)\,\mu(dx)\le c\|\bar R_0\|_{\infty}\mu(D)<\infty$, so %$\MM^+_{0,b}(D)\subset\RR^+(E)$.
It is  well known (see, e.g., \cite{A}) that the transition
density $p$ of the process associated with $(B,V)$ has the
property that $p(t,x,y)\le Ct^{-d/2}$, $t>0$, for some $C>0$, i.e.
(\ref{eq6.8}) with $\alpha=2$ is satisfied. Therefore there is an
analogue of Proposition \ref{prop6.1} for equations involving the
operator $L$ defined by (\ref{eq6.22}).

\vspace{2mm} \noindent{\bf\large Acknowledgements}
\smallskip\\
Research supported by Polish National Science Centre (grant no.
2012/07/B/ST1/03508).


\begin{thebibliography}{32}

\bibitem{A}
Aronson, D.G.: Non--Negative Solutions of Linear Parabolic Equations,
{\em Ann. Scuola Norm. Sup. Pisa} {\bf 22} (1968) 607--693.

\bibitem{BH}
Bass, R.F., Hsu, P.: Some potential theory for reflecting Brownian
motion in H\"older and Lipschitz domains. Ann. Probab. {\bf
19}, 486--508 (1991)

\bibitem{BG}
Boccardo, L.,  Gallou\"et, T.: Strongly nonlinear elliptic
equations having natural growth terms and $L^1$ data.
Nonlinear Anal. {\bf 19}, 573--579 (1992)

\bibitem{BDHPS}
Briand, P., Delyon, B., Hu, Y., Pardoux, E., Stoica, L.: $L^p$
solutions of backward stochastic differential equations.
Stochastic Process. Appl. {\bf 108}, 109--129 (2003)

\bibitem{CKS}
Carlen, E.A., Kusuoka, S., Stroock, D.W.: Upper bounds for
symmetric Markov transition functions. Ann. Inst. H.
Poincar\'e Probab. Statist. {\bf 23} no. 2, suppl., 245--287
(1987)

\bibitem{CF}
Chen, Z.-Q., Fukushima, M.: Symmetric Markov processes, time
change, and boundary theory. Princeton University Press,
Princeton, NJ (2012)

\bibitem{C}
Chung, K.L.: Greenian bounds for Markov processes. Potential
Anal. {\bf 1}, 83--92 (1992)

\bibitem{CZ}
Chung, K.L., Zhao, Z.: From Brownian Motion to
Schr\"odinger's Equation. Springer, Berlin Heidelberg
(1995)

\bibitem{DMOP}
Dal Maso, G., Murat, F., Orsina, L., Prignet, A.: Renormalized
Solutions of Elliptic Equations with General Measure Data.
Ann. Scuola Norm. Sup. Pisa Cl. Sci. (4) {\bf 28}, 741--808
(1999)

\bibitem{DM}
Dellacherie, C., Meyer, P.-A.: Probabilit\'es at potentiel.
Chaptires V \`a VIII. Th\'eorie des martingales. Revised edition.
Hermann, Paris (1980)

\bibitem{FOT}
Fukushima, M., Oshima, Y., Takeda, M.: Dirichlet Forms and
Symmetric Markov Processes.  Second revised and extended edition.
Walter de Gruyter, Berlin (2011)

\bibitem{K:JEE}
Klimsiak, T.: Existence and large-time asymptotics for solutions
of semilinear parabolic systems with measure data. J. Evol.
Equ. {\bf 14}, 913--947 (2014)

\bibitem{K:AMPA}
Klimsiak, T.: Semilinear elliptic systems with measure data.
Ann. Mat. Pura Appl. (4) {\bf 194}, 55--76 (2015)

\bibitem{K:JFA}
Klimsiak, T.: Semi-Dirichlet forms, Feynman-Kac functionals and
the Cauchy problem for semilinear parabolic equations. J.
Funct. Anal. {\bf 268}, 1205--1240 (2015)

\bibitem{KR:JFA}
Klimsiak, T., Rozkosz, A.: Dirichlet forms and semilinear elliptic
equations with measure data. J. Funct. Anal. {\bf 265},
890--925 (2013)

\bibitem{KR:JEE}
Klimsiak, T., Rozkosz, A.: Obstacle problem for semilinear
parabolic equations with measure data. J. Evol. Equ. {\bf
15} (2015) 457--491.

\bibitem{KR:NoD}
Klimsiak, T., Rozkosz, A.: Renormalized solutions of semilinear
equations involving measure data and operator corresponding to
Dirichlet form. NoDEA Nonlinear Differential Equations Appl.
{\bf 22}, 457--491 (2015)

\bibitem{KR:CM}
Klimsiak, T.,  Rozkosz, A.: Semilinear elliptic equations with
measure data and quasi-regular Dirichlet forms. Colloq.
Math. {\bf 145}, 35--67 (2016)

\bibitem{Ku}
Kulczycki, T.: Properties of Green function of symmetric stable
processes. Probab. Math. Statist. {\bf 17},  339--364 (1997)

\bibitem{LP}
Leonori, T., Petitta, F.: Asymptotic behavior for solutions of
parabolic equations with natural growth terms and irregular data.
Asymptot. Anal. {\bf 48}, 219--233 (2006)

\bibitem{LS}
Liptser, R. Sh., Shiryayev, A.N.: Theory of Martingales.
Nauka, Moscow, 1986; English transl. Kluwer, Dordrecht (1989)

\bibitem{MR}
Ma, Z.-M., R\"ockner, M.: Introduction to the Theory of
(Non--Symmetric) Dirichlet Forms. Springer, Berlin (1992)

\bibitem{MV}
Marcus, M., V\'eron, L.:  Nonlinear second order elliptic
equations involving measures. De Gruyter, Berlin (2014)

\bibitem{MP}
Murat, F., Porretta, A.: Stability properties, existence, and
nonexistence of renormalized solutions for elliptic equations with
measure data. Comm. Partial Differential Equations {\bf 27},
2267--2310 (2002)

\bibitem{O1}
Oshima, Y.: Some properties of Markov processes associated with
time dependent Dirichlet forms. Osaka J. Math. {\bf 29},
103--127 (1992)

\bibitem{O2}
Oshima, Y.: Semi-Dirichlet Forms and Markov Processes.
Walter de Gruyter, Berlin (2013)

\bibitem{PZ}
Pardoux, E., Zhang, S.: Generalized BSDEs and nonlinear Neumann
boundary value problems. Probab. Theory Related Fields {\bf
110}, 535--558 (1998)

\bibitem{Pe1}
Petitta, F.: Asymptotic behavior of solutions for linear parabolic
equations with general measure data. C. R. Math. Acad. Sci.
Paris {\bf 344}, 535--558 (2007)

\bibitem{Pe2}
Petitta, F.: Asymptotic behavior of solutions for parabolic
operators of Leray-Lions type and measure data. Adv.
Differential Equations {\bf 12}, 867--891 (2007)

\bibitem{Pe3}
Petitta, F.: Large time behavior for solutions of nonlinear
parabolic problems with sign-changing measure data. Electron.
J. Differential Equations No. 132, 10 pp. (2008)

\bibitem{PPP}
Petitta, F., Ponce, A.C., Porretta, A.: Diffuse measures and
nonlinear parabolic equations. J. Evol. Equ. {\bf 11},
861--905 (2011)

\bibitem{Pi}
Pierre, M.: Representant Precis d'Un Potentiel Parabolique,
Seminaire de Theorie du Potentiel. Lecture Notes in Math.
{\bf 814}, 186--228 (1980)

\bibitem{Po}
Porretta, A.: Existence results for nonlinear parabolic equations
via strong convergence truncations. Ann. Mat. Pura Appl.
(IV) {\bf 177}, 143--172 (1999)

\bibitem{Pr}
Protter, P.: Stochastic Integration and Differential
Equations. Second Edition. Springer, Berlin (2004)

\bibitem{S}
Stannat, W.: The Theory of Generalized Dirichlet Forms and Its
Applications in Analysis and Stochastics. Mem. Amer. Math.
Soc. {\bf 142}, no. 678, viii+101 pp. (1999)

\end{thebibliography}
\end{document}